\documentclass[11pt, reqno]{amsart}
\usepackage{enumerate}
  \usepackage{geometry }
 \usepackage{amsmath}
 \usepackage{amsfonts}
 \usepackage{amssymb}
  \usepackage{enumerate}
 
\usepackage{amssymb,amsmath,amscd}
\usepackage{amsthm}
 \usepackage[all]{xy}
 
\usepackage{color}
 \usepackage{mathtools}
 
\newcommand*\diff{\mathop{}\!\mathrm{d}}

 \numberwithin{equation}{section}

   \newtheorem*{thmalph}{Theorem}

\newtheorem{theorem}{Theorem}[section]
 
\newtheorem*{theorem*}{Theorem }
\newtheorem{proposition}[theorem]{Proposition}

\newtheorem{corollary}[theorem]{Corollary}

\newtheorem{lemma}[theorem]{Lemma}

\DeclareMathOperator{\tr}{tr}

   \DeclareMathOperator{\ib}{\boldsymbol  \iota}
   
 \usepackage{mathtools}
\usepackage[bbgreekl]{mathbbol}
\usepackage{amsfonts}

\DeclareSymbolFontAlphabet{\mathbb}{AMSb}
 \usepackage[
 colorlinks=true,  citecolor=cyan,   
 citebordercolor={0 .255 .255} , 
  linkbordercolor={0 .255 .255},  
  linkcolor=cyan
 ]{hyperref}

   \usepackage{cite}

 \def\equationautorefname~#1\null{(#1)\null}

\usepackage{multirow}
  \usepackage{graphicx}
 \usepackage{caption}
 \usepackage{booktabs}
 \usepackage{threeparttable}
\newcommand{\C}{\mathbb C}

\renewcommand{\det}{\mathbf{det}}
\renewcommand{\tr}{\mathbf{tr}}

      \DeclareMathOperator{\SO}{SO}

\makeatletter
\newcommand{\extp}{\@ifnextchar^\@extp{\@extp^{\,}}}
\def\@extp^#1{\mathop{\bigwedge\nolimits^{\!#1}}}
\makeatother

      \usepackage[normalem]{ulem}
      \usepackage{color}
        \usepackage{xcolor}
      \usepackage{cancel}
      \usepackage{slashed}

\newcommand\blfootnote[1]{%
  \begingroup
  \renewcommand\thefootnote{}\footnote{#1}%
  \addtocounter{footnote}{-1}%
  \endgroup
}

\begin{document}

\title[Poisson transform for differential forms]
{On Poisson transforms  of differential forms on real hyperbolic spaces}
\author{Salem Ben Said}
\address{Salem Ben Said : 
Mathematical Sciences Department, College of Science, United Arab Emirates University, Al Ain, UAE
}
 
\email{salem.bensaid@uaeu.ac.ae}
\thanks{}

\author{Abdelhamid Boussejra}
\address{Abdelhamid Boussejra : Department of Mathematics, Faculty of Sciences, University Ibn Tofail, Kenitra, Morocco 
}
\curraddr{}
\email{boussejra.abdelhamid@uit.ac.ma}
\thanks{}

\author{Khalid Koufany}
\address{Khalid Koufany : Université de Lorraine, CNRS, IECL, F-54000 Nancy, France}

\curraddr{}
\email{khalid.koufany@univ-lorraine.fr}
\thanks{}
\date{\today}
 
 \maketitle

\begin{abstract} This paper is concerned with the Poisson transform of differential forms on the hyperbolic space $H^n(\mathbb R)$.
Consider an integer $p$ such that $1\leqslant p\leqslant n$   and let $q$ be either $p-1$ or $p$. 
For $1<r<\infty$, we prove that the Poisson transform is a topological isomorphism from the space of $L^r$-differential $q$-forms on the boundary $\partial  H^n(\mathbb R)$  onto a Hardy-type subspace of $p$-eigenforms of the Hodge-de Rham Laplacian on $H^n(\mathbb R)$.  

\end{abstract}
\blfootnote{\emph{2010 AMS Classification : \rm 58A10, 43A85, 53C35, 22E30}} 
\tableofcontents

\section{Introduction}
Let  $G/K$ be a Riemannian symmetric space of non-compact type and let $G/P$ be its Furstenberg boundary. In \cite{Helgason-70}, S. Helgason claimed that all eigenfunctions of $G$-invariant differential operators on $G/K$ are Poisson transforms of hyperfunctions on $G/P.$  This conjecture was initially proved in \cite{Helgason-70}  for  the case where  $G/K$ has rank one, and was later established in its entirely by Kashiwara {\it et al.} \cite{KKMOOT}.  Since then, this conjecture  has received  significant attention across various settings (see, e.g., \cite{BK, 1, 2, 3, 4, 5, 6, 7, 8, 9, Schmid,  Shimeno}).

This problem has been extended to  include Poisson transforms for homogeneous vector bundles over  $G/K$, (see, e.g., \cite{Gaillard, Gaillard2,   Juhl, olbrich, van, Yan}).   
  In this paper, we will focus on studying the Poisson transform for   the bundle of differential forms on the real hyperbolic space $H^n(\mathbb R)$.

To elaborate further,  we will suppose $G=\SO_{\rm o}(n,1)$, $K= \SO(n)$ and $H^n(\mathbb R)=G/K$. The boundary is realized as $\partial H^n(\mathbb R)=G/P$ where  $P=MAN$  with $M=\SO(n-1)$, $A\simeq \mathbb R$ and $N\simeq \mathbb R^{n-1}$.

 Let $0\leqslant p\leqslant n$ be an integer, and consider the space  $C^\infty(\extp^p H^n(\mathbb R ))$ of smooth differential $p$-forms on the hyperbolic space, {\it i.e.} smooth sections of the bundle $\extp^p H^n(\mathbb R):=\extp^p T^*_{\mathbb C}H^n(\mathbb R ))$.

   Let  $\mathbb D\left(\extp^pH^n(\mathbb R\right)$  be the algebra of letf-invariant differential operators, acting on  the space  $C^\infty(\extp^p H^n(\mathbb R ))$.  To keep this introduction simple for the reader, we suppose $p$ generic, i.e. $1 \leqslant p<\frac{n-1}{2}$ (but we get similar results  for general $p$).  We exclude the case 
  $p=0$, which corresponds to the well-known case of functions.  
   A result of Gaillard \cite{Gaillard3} states that  in this case $\mathbb D\left(\extp^p(H^n(\mathbb R)\right)$  is a commutative algebra generated by $dd^*$ and $d^*d$, where $d$ is 
 exterior differentiation operator and $d^*$ is the  co-differentiation operator. Let us further denote $\Delta=dd^*+d^*d$ the Hodge de-Rham operator.
 Therefore, any character $\chi$ of $\mathbb D\left(\extp^pH^n(\mathbb R)\right)$  can be only of type    $\chi_{p,\lambda}(dd^*) =\lambda^2 + (\rho-p)^2$ or  $\chi_{p-1,\lambda}(d^*d) =\lambda^2 + (\rho-p+1)^2$, where $\rho=\frac{n-1}{2}$.  Let 
 \begin{equation}\label{1}
 	\mathcal E_{q,\lambda}^p=\left\{ F\in C^\infty\left(\extp^pH^n(\mathbb R)\right) \; |\;  DF=\chi_{q,\lambda}(D) F,\;\;  \forall D\in\mathbb D\left(\extp^p H^n(\mathbb R)\right)\right\}, \; q=p, p-1
 \end{equation}

Note that $\mathcal{E}_{p, \lambda}^p$ (resp. $\mathcal{E}_{p-1, \lambda}^p$ ) represents the space of coclosed (resp. closed) eigenforms of $\Delta$ associated with the eigenvalue $\lambda^2+(\rho-p)^2$ (resp. $\lambda^2+(\rho-p+1)^2$ ). Gaillard \cite{Gaillard, Gaillard2} proved that for regular values of $\lambda$, there exist $G$-isomorphisms,
$$
C^{-\omega}_{q} \xrightarrow {\sim} \mathcal E_{q,\lambda}^p 
$$
realized by Poisson transformations $\mathcal P_{q,\lambda}^p$, where $C^{-\omega}_q=C^{-\omega}\left(\extp^q T_{\mathbb C}^*\partial H^n(\mathbb R)\right)$ denotes the space of hyperfunction vectors of an induced representation from $P=M A N$ to $G$, constructed using the natural representation of $M$ on $\extp^q \mathbb{C}^{n-1}$, a suitable character of $A$, and the trivial representation of $N$ (see Theorem \ref{gaillard}, Corollary \ref{bijective1} and Corollary \ref{cor-olbrich}).

The purpose of this paper is to study the Poisson transforms on the space $L^r\left(\extp^q T_{\mathbb C}^* \partial H^n(\mathbb{R})\right)$, for $1<r<\infty$, and to precisely characterize its  image. Let us now provide further details.

  Let  $\tau_{p}$ be the $p$-th exterior power   of  the coadjoint representation  of $K$ acting on  
$V_{\tau_p}:= \bigwedge^{p} (\mathfrak g_{\mathbb C}/ \mathfrak k_{\mathbb C})^*\simeq\bigwedge^{p} \mathbb{C}^{n},$ with the associated  vector bundle  $  G\times_K V_{\tau_p}$.  
Since $\extp^p H^n(\mathbb R)$ is canonically isomorphic to   $G\times_K V_{\tau_p}$, we identify $C^\infty (\extp^p H^n(\mathbb R)) $ with the space 
  $C^{\infty}(G/K; {\tau_p})$ of  smooth $\tau_p$-equivariant  functions $f : G\to V_{\tau_p}$.

 With our assumption, $p$ is generic ($1 \leqslant p<\frac{n-1}{2}$),  the representation $\tau_p$ is irreducible and    the restriction of  $\tau_p$ to $M$  decomposes as $ {\tau_{p}}_{|M}=\sigma_{p-1}\oplus \sigma_p,
 $ where the representation  $\sigma_q$, for $q=p-1,p$,  is the $q$-th exterior power   of  the coadjoint representation  of  $M$ on 
 $V_{\sigma_q}\simeq  \bigwedge^{q} \mathbb{C}^{n-1}.$  Denote by $\widehat{M}(\tau_p)$ the set of irreducible components of ${\tau_p}_{|M}$. So in this case (generic $p$), $\widehat{M}(\tau_p)=\{\sigma_{p-1},\sigma_p\}$ (see \eqref{MMM} for the general case, $1\leqslant p\leqslant n$).

 Denote by   $\mathfrak{a}$   the Lie algebra of $A.$  
 For $\lambda\in\mathfrak a^\ast_\mathbb{C}$ and $\sigma\in  \widehat{M}(\tau)$, we consider the representation $ \varpi_{\sigma, \lambda}$ of $P=MAN$ acting  on $V_{\sigma}$ by 
 $ \varpi_{\sigma, \lambda}(ma_tn)=e^{(\rho-i\lambda)t} \sigma(m)$ where $m\in M$, $a_t\in A$, $n\in N$. Let  $F_{\varpi_{\sigma,\lambda}}$ be the associated homogeneous  vector bundle over $G/P,$  and let  
 $C^{-\omega}(F_{\varpi_{\sigma,\lambda}}) $ be the space of its hyperfunctional sections, which will be identified with the space  $C^{-\omega}(G/P;\varpi_{\sigma,\lambda})$  of    hyperfunctions $f: G \to V_{\sigma}$  satisfying 
$
f(gma_tn)={\rm e}^{(i\lambda-\rho)t}\sigma(m^{-1})f(g)$ for every $g\in G$, $m\in M$, $n\in N$ and  $a_t\in A$. 
 Let $C^{-\omega}({T^*_\mathbb{C}}\extp^p \partial H^n(\mathbb R))$ be the space of $p$-forms with hyperfunction coefficients on $\partial H^n(\mathbb R^n)$. Then we have
$ C^{-\omega}({T^*_\mathbb{C}}\extp^p \partial H^n(\mathbb R)) \simeq 
C^{-\omega}(K/M, \sigma)$, the space of hyperfunctions $F: K\to V_\sigma$ such that $F(km)=\sigma(m)^{-1}F(k)$ ($m\in M, k\in K$), see background for more details.

Now, using the fact that, as a $K$ module,    $C^{-\omega}(G/P;\varpi_{\sigma, \lambda})$ is   isomorphic to the space $C^{-\omega}\left(K/M;{\sigma}\right)$,  we define (in the compact picture) the 
Poisson transform of any  hyperform $f\in  C^{-\omega}\left(K/M; {\sigma}\right)$  by \begin{align*}
\mathcal P^{\tau}_{\sigma,\lambda}f(g)=\sqrt{\frac{\dim \tau}{\dim\sigma}}  \int_K  {\rm e}^{-(i\lambda+\rho)H(g^{-1}k)}\tau(\kappa(g^{-1}k)) \ib_\sigma^\tau(f(k)){\rm d}k,\;\; g\in G,
\end{align*}
where $\ib_\sigma^\tau$ is  the natural   embedding of $V_{\sigma}$ into $V_{\tau}$.

Taking the above identifications into account, we regard the algebra $\mathbb{D}\left(\extp^p H^n(\mathbb{R})\right)$ as the algebra $\mathbb{D}\left(G / K; \tau\right)$ of $G$-invariant differential operators acting on the space $C^{\infty}\left(G / K; \tau\right)$. Furthermore, we continue to treat $d d^*$ and $d^* d$ as the fundamental generators (for generic $p$).

The previously mentioned eigenspaces  \eqref{1} are now  (for $p$ generic)
 $$\mathcal{E}_{\sigma_p,\lambda}(G/K; {\tau_p})=\{F\in C^\infty(G/K; \tau_p) : \; \Delta F=(\lambda^2+(\rho-p)^2)F \,\text{ and }\;  d^* F=0\} ,$$
and 
$$\mathcal{E}_{\sigma_{p-1},\lambda}(G/K; {\tau_p})=\{F\in C^\infty(G/K; \tau_p) : \; \Delta F=(\lambda^2+(\rho-p+1)^2)F \,\text{ and }\;  dF=0\} .$$

For a general integer $1 \leqslant p \leqslant n$, the Hodge $\star$-operator implies that $\tau_p$ and $\tau_{n-p}$ are equivalent, allowing us to restrict our focus to $1 \leqslant p \leqslant \frac{n}{2}$. It is well known that $\tau_p$ is irreducible if and only if $p \neq \frac{n}{2}$. Moreover, when $p=\frac{n}{2}$, $\tau_{\frac{n}{2}}$ decomposes into two irreducible subrepresentations, $\tau_{\frac{n}{2}}=$ $\tau_{\frac{n}{2}}^{+} \oplus \tau_{\frac{n}{2}}^{-}$, where $\tau_{\frac{n}{2}}^{ \pm}$ correspond to the eigenspaces of the Hodge $\star$-operator and its restriction to $M$ is given by  ${\tau_{\frac{n}{2}}^{ \pm}}_{|M}=\sigma_{\frac{n}{2}}$. Furthermore, for $p=\frac{n-1}{2}$, we have ${\tau_{\frac{n-1}{2}}}_{|M}=\sigma_{\frac{n-1}{2}-1} \oplus \sigma_{\frac{n-1}{2}}^{+} \oplus \sigma_{\frac{n-1}{2}}^{-}$. Based on this branching law, for pairs $(\tau, \sigma)=(\tau_p^{( \pm)}, \sigma_q^{( \pm)})$ with $q=p$ or $p-1$, we define the Poisson transforms $\mathcal{P}_{\sigma, \lambda}^\tau$ and the eigenspaces $\mathcal{E}_{\sigma, \lambda}\left(G / K; \tau\right)$, in analogy with the generic case. For a full description, refer to Section 3.

For $1< r< \infty,$ let   $L^r\left(K/M;{\sigma}\right) $ be the subspace of $L^r$-hyperforms in $ C^{-\omega}\left(K/M; {\sigma}\right)$ and recall that     the central goal of this paper is to characterize the image $\mathcal{P}_{\sigma, \lambda}^\tau\left(L^r\left(K/M; {\sigma}\right)\right)$.   To accomplish this,   we introduce  the space     $\mathcal{E}^{r}_{\sigma,\lambda}(G/K; {\tau})$    of all $F$ in  $ \mathcal E_{\sigma,\lambda}(G/K; {\tau}) $ satisfying 
\begin{equation*} 
\Vert F\Vert_{{r,\lambda}}:=\sup_{t>0}{\rm e}^{ (\rho-\Re(i\lambda))t}\left(\int_K \Vert F(ka_t) \Vert_{V_{\tau}}^r {\rm d}k\right)^\frac{1}{r} <\infty,
\end{equation*}
where $\diff k$ is the normalized Haar measure of $K$.  The main result  is: 

\begin{thmalph}[See Theorem \ref{cas-Lr}]\label{thmA}
Let $\tau =\tau_1,\cdots,\tau_{\frac{n-1}{2}}, \tau^\pm_{\frac{n}{2}}$ and $\sigma=\sigma_q\in  \widehat{M}(\tau)$ accordingly.  Consider 
  $\lambda\in \mathbb C$ such that 
  \begin{equation*} 
\begin{cases}
\Re(i\lambda)>0,&\text{ if $q=p$}\\ 
 \Re(i\lambda)>0, \, \text{ and } \; i\lambda\not=\rho-p+1& \text{ if $q=p-1$}
 \end{cases}
 \end{equation*}
 where $\rho=\frac{n-1}{2}.$
 Then the Poisson transform $\mathcal{P}_{\sigma,\lambda}^{\tau}$ is  a topological isomorphism
from the space $L^r\left(K/M; {\sigma}\right)$ onto the space $\mathcal{E}^r_{\sigma,\lambda}(G/K; {\tau})$.   Moreover,  for every $f\in L^r\left(K/M;{\sigma}\right)$,  we have 
\begin{equation*}\label{esti-F}
  | c_\sigma(\lambda,\tau)| \;\| f\|_{L^r\left(K/M; {\sigma}\right)}\leq \sqrt{\frac{\dim \sigma}{\dim \tau}} \;
 \|\mathcal{P}_{\sigma,\lambda}^{\tau} f\|_{{r,\lambda}} \leq    \gamma_\lambda \| f\|_{ L^r\left(K/M; {\sigma}\right)},
\end{equation*}
 where $\gamma_\lambda$  is a positive constant and  $c_\sigma(\lambda,\tau)$  denotes the scalar component of the vector-valued Harish-Chandra $c$-function $\mathbf c(\lambda,\tau).$  Refer to 
Proposition \ref{c-explicit} for its explicit expression.
\end{thmalph}

As an immediate consequence of the main theorem, we obtain a characterization of coclosed harmonic $p$-forms via the Poisson transform, corresponding to the following description:  generic $p$, $q=p$ and $i \lambda=\rho-p$ (see Corollary \ref{corr-harm}).
Moreover, if we additionally assume $p=0$, we recover the well-known result that the Poisson transform is an isometric isomorphism from $L^r\left(\extp^p \partial {H}^n(\mathbb R)\right)$ onto a Hardy-harmonic space on ${H}^n(\mathbb R)$ (see \cite{stoll}).

The paper is organized as follows: Section \ref{sec2} provides the necessary background, while Section \ref{sec3} introduces the Poisson transform on the space of differential forms on the boundary $\partial H^n(\mathbb{R})$. In Section \ref{sec4}, we establish a Fatou-type theorem for Poisson integrals, 
a crucial tool in obtaining the explicit formulas for the scalar component
  $c_\sigma(\lambda, \tau)$ that appear in the main result. Section \ref{sec5} makes fundamental use of this Fatou-type theorem to establish the main result for $r=2$, and also provides an $L^2$-inversion formula for the Poisson transform. Finally, in Section \ref{sec6}, we synthesize all preceding   results to prove the main theorem for all $1<r<\infty$.

 \section{Background}\label{sec2}

\subsection{The real hyperbolic space}    
    Let $H^{n}(\mathbb{R})$ be the $n$-dimensional real hyperbolic space ($n \geqslant 2$)
  viewed as the rank one symmetric space of the noncompact type $G/K$ where  $G=\SO_0(n,1)$ and $K=\SO(n)$.

 Let  $\mathbf {\mathfrak g}\simeq\mathfrak{so}(n,1)$ and $\mathbf {\mathfrak k}\simeq\mathfrak{so}(n)$ be the Lie algebras of $G$ and $K$ respectively  and write $\mathfrak g=\mathfrak k\oplus \mathfrak p$ for the Cartan decomposition of $\mathfrak g$.
 The tangent space $T_o(G / K) \simeq   \mathfrak{p}$ of $G / K=H^n(\mathbb{R})$ at the origin $o=e K$ will be identified with the vector space $\mathbb{R}^n$.

  There exists an element $H_0\in\mathfrak p$ such that 
$\mathfrak{a}=\mathbb{R} H_0$ is a Cartan subspace in $\mathfrak{p}$. Let  $A=\exp\mathfrak a=\{a_t= e^{tH_0},\; t\in\mathbb R\}$ be the corresponding analytic Lie subgroup of $G$.
 We define  $\alpha \in \mathfrak{a}^*$ by $\alpha\left( H_0\right)=1$.  
Then the positive restricted root subsystem is $\Sigma^{+}(\mathfrak{g}, \mathfrak{a})=\{\alpha\}$.    
 We will identify $\mathfrak a_{\mathbb C}^*$ and $\mathbb C$ via the map $\lambda \alpha \mapsto \lambda$. Then the half-sum of positive roots is $\rho=(n-1)\alpha/2 = (n-1)/2$.

Let $\mathfrak{n}=\mathfrak{g}_\alpha\simeq \mathbb R^{n-1}$ be the   positive root subspace and $N=\exp(\mathfrak n)$ the corresponding analytic subgroup  of $G$.   The groupe $G$ has an Iwasawa decomposition $G=KAN$. Thus, each $g\in G$ can be uniquely written as 
$
g=k(g) e^{H(g)} n(g)$, where $k(g)\in K$, $H(g)\in \mathfrak{a}$ and   $n(g)\in N
$. 
 
Let $P=MAN$ be the standard minimal parabolic subgroup of $G$, where $M=\SO(n-1)$ is the centralizer  of $A$ in $K.$
  Then the boundary $\partial H^n(\mathbb R)$ is   realized as the flat space $\partial H^n(\mathbb R)=G/P= K/M$.

 \subsection{Differential forms on $H^n(\mathbb R)$ and $\partial H^n(\mathbb R)$} 
  Let $\langle \cdot,\cdot\rangle$ be the standard Euclidean scalar product in $\mathbb R^n$. Let 
 $(e_1,e_2,\ldots, e_n)$ be the standard orthonormal  basis of  $\mathbb{R}^n$   and denote $(e_1^*,e_2^*,\ldots, e_n^*)$ its dual basis. 
 
 For an integer $p$ such that $0\leqslant p\leqslant n$,   let $ \bigwedge^p (\mathbb C^{n})^*= \bigwedge^p(\mathbb R^{n})^*\otimes\mathbb{C}$   be the space of complex-valued  alternating multilinear $p$-forms on $\mathbb R^n.$ 
  
   For the reader's convenience and to keep the notations simple, we will identify $(\mathbb C^n)^*$ with $\mathbb C^n$ and $ \bigwedge^p (\mathbb C^{n})^*$ with $ \bigwedge^p \mathbb C^{n}$.
 
 We define an inner  product $\langle \cdot, \cdot\rangle_{\bigwedge^p\mathbb C^n}$ on $\bigwedge^p\mathbb C^n$ as an extension of the one   on 
 $\mathbb C^n$   by setting
 \begin{equation}\label{ps-p-vectors}
\langle v_1\wedge\cdots \wedge v_p,w_1\wedge\cdots\wedge w_p\rangle_{\bigwedge^p\mathbb C^n}=\det (\langle v_i,w_j\rangle)_{i,j}.
 \end{equation}
It is easy to show that the basis of $\bigwedge^p\mathbb C^n$ consisting of the $p$-vectors  $e_I:=e_{i_1}\wedge \cdots \wedge e_{i_p}$, where $ I=\{i_1,\cdots,i_p\}$, with $1\leqslant i_1<\cdots <i_p\leqslant n$,  is an orthonormal basis of 
 $\bigwedge^p\mathbb C^n $ with respect to the inner product \autoref{ps-p-vectors}.

 In this section, we  choose and fix   an integer  $p$, such that $0\leqslant p\leqslant n$. Let  $\extp^p H^n(\mathbb R):= \extp^p T^*_{\mathbb C}H^n(\mathbb R)$ be the $p$-th exterior power of the complexified cotangent bundle of $H^n(\mathbb R)$. 
A differential $p$-form on $H^n(\mathbb R)$ is a section on $\extp^p H^n(\mathbb R)$.

  Let $\tau_{p}$ be the standard representation of $K$ on $V_{\tau_p}=\extp^{p} \mathbb{C}^n$. Notice that $\tau_p$ is equivalent to the $p$-th exterior power   of  the coadjoint representation $Ad^\ast$ of $K$ on 
 $\mathfrak{p}_{{\mathbb C}}^\ast$. Let $G\times_K V_{\tau_p}$ be the the $G$-homogeneous vector bundle   associated to $\tau_p$. Since $T_{eK} H^n(\mathbb{R})\simeq \mathfrak{p}\simeq \mathbb{R}^n$, then as $G$-homogeneous vector bundles we have   $\extp^p H^n(\mathbb{R})\simeq G\times_K V_{\tau_p}$. As usual we shall identify the space  $C^\infty( \extp^p H^n(\mathbb{R}))$ of smooth differential  $p$-forms on  $H^n(\mathbb{R})$ with the space $C^\infty(G/K;\tau_p)$  of  smooth functions   $F:G\to V_{\tau_p}$ which are right $K$-covariant, i.e., 
\begin{equation}\label{t-equiv16}
   	F(g k)={\tau_p}\left(k^{-1}\right) F(g) \quad \text{for all $g\in G$,  $k\in K$}.
   \end{equation}

 It is known that the representation $\tau_p$ is irreducible unless $n$ is even and $p=\frac{n}{2}$ in which case it decomposes as 
$\tau_{{n}/{2}}=\tau_{{n}/{2}}^+\oplus \tau_{{n}/{2}}^-$ with the corresponding  decomposition of the representation space $\extp^{\frac{n}{2}} \mathbb{C}^n=\extp_{+}^{\frac{n}{2}} \mathbb{C}^n \oplus \extp_{-}^{\frac{n}{2}} \mathbb{C}^n$, where 
$
\extp^{\frac{n}{2}}_{\pm} \mathbb{C}^n=\{\alpha\in  \extp^{\frac{n}{2}} \mathbb{C}^n; \star \alpha=\mu_{\pm}\alpha\}
$. In this case we have 
$
\extp^{\frac{n}{2}} H^n(\mathbb{R})=\extp_+^{\frac{n}{2}} H^n(\mathbb{R})\oplus \extp_-^{\frac{n}{2}} H^n(\mathbb{R})
$
with $\extp_\pm^p H^n(\mathbb{R})=G\times_K \extp_{\pm}^{\frac{n}{2}} \mathbb{C}^n$. 
Here $\star$ is  the Hodge operator and $\mu_\pm=\pm 1$  if $\frac{n}{2}$ is even and $\mu_\pm=\pm i$ if $\frac{n}{2}$  is odd. Notice that the Hodge operator induces the equivalence $\tau_p\sim \tau_{n-p}$, hence, hereafter we shall restrict our discussion to $0\leqslant p\leqslant \frac{n}{2}$.  

To distinguish between   representations of $K$ and   representations of $M$, we will 
  use the Greek letter   $\sigma$ to denote  those representations of    $M$. This notation will help clearly differentiate the representations associated with these two groups throughout our discussion, ensuring that the analysis remains precise and unambiguous.
  
For $0\leqslant q\leqslant n-1$, let  $\sigma_q$ be the standard representation of $M$ on $V_{\sigma_q}=\extp^q(\mathbb Ce_2\oplus\cdots\oplus\mathbb C e_n)=\extp^q\mathbb C^{n-1}$, where $(e_j)_{j=1}^n$ is the natural basis of $\mathbb C^n$. Then
 the  branching rules for $(K,M)=(\SO(n),\SO(n-1))$ is given as follows (see e.g. \cite{Pedon, BS, IT}) :
 \begin{lemma}\label{lema2.1}
 	\begin{enumerate}
\item[$(1)$] if $p< \frac{n-1}{2}$, then ${\tau_p}_{|M}=\sigma_{p-1}\oplus \sigma_{p}$, with
   \begin{equation*} 
\textstyle  \bigwedge^p\mathbb{C}^n
=e_1\wedge \textstyle\bigwedge^{p-1}\mathbb{C}^{n-1}\oplus \bigwedge^{p}\mathbb{C}^{n-1}
\simeq  \bigwedge^{p-1}\mathbb{C}^{n-1}\oplus \bigwedge^{p}\mathbb{C}^{n-1};
  \end{equation*}
\item[$(2)$] if $p=\frac{n-1}{2}$, then ${\tau_{p}}_{|M}=\sigma_{p-1}\oplus\sigma_{p}^+\oplus \sigma_{p}^-$, with
\begin{equation*} 
 \textstyle \bigwedge^p\mathbb{C}^n
 =e_1 \wedge \bigwedge^{p-1} \mathbb{C}^{n-1} \oplus \bigwedge_{+}^p \mathbb{C}^{n-1} \oplus \bigwedge_{-}^p \mathbb{C}^{n-1}
 \simeq \bigwedge^{p-1} \mathbb{C}^{n-1} \oplus \bigwedge_{+}^p \mathbb{C}^{n-1} \oplus \bigwedge_{-}^p \mathbb{C}^{n-1};
  \end{equation*} 
\item[$(3)$] if $p=\frac{n}{2}$, then ${\tau_{\frac{n}{2}} }_{\big|M}={\tau_{\frac{n}{2}}^{+}}_{\big|M} \tilde{\oplus} {{\tau_{\frac{n}{2}}^{-}}}_{\big|M}=\sigma_{\frac{n}{2}-1} \tilde{\oplus} \sigma_{\frac{n}{2}}$, with
$$ \textstyle 
\bigwedge^{\frac{n}{2}} \mathbb{C}^n
=e_1 \wedge  \bigwedge^{\frac{n}{2}-1} \mathbb{C}^{n-1}  \tilde{\oplus} \bigwedge^{\frac{n}{2}} \mathbb{C}^{n-1}
\simeq \bigwedge_{+}^{\frac{n}{2}} \mathbb{C}^n  \tilde{\oplus} \bigwedge_{-}^{\frac{n}{2}} \mathbb{C}^n.
$$
In particular, 
${\tau_{\frac{n}{2}}^\pm}_{|M}=\sigma_{\frac{n}{2}}$.
  \end{enumerate}
 \end{lemma}
  Above, we used the tilde symbol in the direct sum $R_1 \tilde{\oplus} R_2$ to indicate that the representations $R_1$ and $R_2$ are equivalent. 
  
 We will refer to   {\it generic case} for  $1 \leqslant p \leqslant \frac{n}{2}$ with $p\neq \frac{n-1}{2}, \frac{n}{2}$ (we will not deal with the well-known case $p=0$) and    {\it special cases} for the cases where $n$ is odd and $p=\frac{n-1}{2}$ or $n$ even and $p=\frac{n}{2}$.

For $\tau=\tau_1, \cdots,     \tau_{\frac{n-1}{2}},\tau^\pm_{\frac{n}{2}}$, we denote by  $\widehat{M}(\tau)$ the set of representations in $\widehat{M}$ that occur   in the restriction of $\tau$ to $M$.   According to  the above branching rules we have, 
\begin{equation}\label{MMM}
   	\widehat M(\tau)=\begin{cases}
   	\widehat M(\tau_p)=\{\sigma_{p-1},\sigma_{p}\} & \text{ for $p<\frac{n-1}{2}$},\\
   	\widehat M(\tau_p)=\{\sigma_{p-1},\sigma_{p}^+,\sigma_p^-\} & \text{ for  $p=\frac{n-1}{2}$},\\
   	\widehat M(\tau_p^\pm)=\{\sigma_{p}\} & \text{ for  $p=\frac{n}{2}$.}
   	\end{cases}	
   \end{equation}
  
 For  $1\leqslant p\leqslant \frac{n}{2}$ and $\sigma_q\in \widehat M(\tau_p), $   the pair of representations $(\tau_p, \sigma_q) $  will stand for 
\begin{equation}
\label{qp}
(\tau_p,\sigma_p)
=\begin{cases}
	(\tau_p,\sigma_p) & \text{if $p$ is generic}\cr
	(\tau_{\frac{n-1}{2}},\sigma^\pm_{\frac{n-1}{2}}) & \text{if $p=\frac{n-1}{2}$  }\cr
	(\tau_{\frac{n}{2}}^\pm,\sigma_{\frac{n}{2}}) & \text{if $p=\frac{n}{2}$ }\cr
\end{cases}
\end{equation}
while  the pair $(\tau_p, \sigma_{p-1})$ represents 
\begin{equation}
\label{qp-1}
(\tau_p,\sigma_{p-1})
=\begin{cases}
	(\tau_p,\sigma_{p-1}) & \text{if $p$ is generic}\\
	(\tau_{\frac{n-1}{2}},\sigma_{\frac{n-1}{2}-1}) & \text{if $p=\frac{n-1}{2}$ }\\
\end{cases}
\end{equation}
  
Let $\sigma\in \widehat{M}(\tau)$. Since $T_{eM} K/M=\mathfrak k/\mathfrak m\simeq \mathfrak a^\perp\simeq \mathbb R^{n-1}$ then   the space 
$\bigwedge^q \partial H^n(\mathbb R):=\bigwedge^q T^*_{\mathbb C}\partial H^n(\mathbb R)$ of differential $q$-forms on $\partial H^n(\mathbb R)$ is canonically isomorphic to the homogenous bundle  $K\times_M V_{\sigma}=K\times_M \bigwedge^q \mathbb C^{n-1}$, with $q=p, p-1$.   

For $1<r<\infty$, the space $L^r(\bigwedge^q \partial H^n(\mathbb R))$   of $L^r$ $q$-forms   on $\partial H^n(\mathbb R)$ will be identified with
the space  $L^r(K/M;\sigma)$    of vector valued functions   $f:K\to V_\sigma$ satisfying the identity
\begin{equation}\label{aout8}
	f(km)=\sigma(m)^{-1}f(k),\;  \text{for all $k\in K,\ m\in M$}
\end{equation}
and such that
$$\parallel f\parallel_{L^r(K/M;\sigma)}=\left( \int_K \parallel f(k)\parallel^r\, {\rm d}k\right)^{\frac{1}{r}}<\infty.
$$ 
Similarly, we will  identify the space $C^{-\omega}(\bigwedge^q \partial H^n(\mathbb R))$ of $q$-hyperforms on $\partial H^n(\mathbb R)$ with the space $C^{-\omega}(K/M;\sigma)$ of vector-valued hyperfunctions $f:K\to V_\sigma$ satisfying the identity \eqref{aout8}.

  \section{Poisson transform on differential forms}\label{sec3}

 We will now introduce the  Poisson transform  for differential forms on $\partial H^n(\mathbb{R})$.  
Let  $\tau =\tau_1,\cdots,\tau_{\frac{n-1}{2}},   \tau_{\frac{n}{2}}, \tau^\pm_{\frac{n}{2}}$ and $\sigma\in \widehat{M}(\tau)$. Let $\ib_\sigma^\tau : V_\sigma\to V_\tau$ the natural embedding.  
    For any $\lambda\in\mathbb C$, the Poisson transform
    $$
\mathcal{P}^{\tau}_{\sigma,\lambda}: C^{-\omega}(K/M;\sigma)\to C^\infty(G/K;\tau)
$$ 
is given by
\begin{align}\label{pois}
\mathcal{P}^{\tau}_{\sigma,\lambda}f(g)=  {d_{\tau,\sigma}}
 \int_K \diff k\,{\rm e}^{-(i\lambda+\rho)H(g^{-1}k)}  \tau(\kappa(g^{-1}k)) \, \ib_\sigma^\tau    f(k),
\end{align}
 where  
 \begin{equation}\label{d26sept}
 d_{\tau,\sigma}=\sqrt{\frac{\dim \tau}{\dim \sigma}}.
 \end{equation}

Notice that, for $\tau_p\in \{\tau_1,\cdots,\tau_{\frac{n-1}{2}},   \tau_{\frac{n}{2}}\}$ and $\sigma=\sigma_p$, the Poisson transform  $\mathcal{P}^{\tau_p}_{\sigma_p,\lambda}$ is up to a constant the Poisson transform $\Phi_p^{\rho-i\lambda}$ investigated by Gaillard in \cite{Gaillard}.

To present the main result from \cite{Gaillard}, let us review the description of     the algebra $\mathbb D(G/K;\tau_p)$ of $G$-invariant differential operators acting on differential forms on $H^n(\mathbb R)$. 

Let $d:C^\infty\!\! \bigwedge^p H^n(\mathbb R) \rightarrow C^\infty\!\! \bigwedge^{p+1} H^n(\mathbb R)$ is the exterior differentiation operator, $d^* =(-1)^{n(p+1)+1} \star d\, \star$ is the  
co-differentiation operator, and $\Delta=dd^*+d^*d$ is the Hodge-de Rham Laplacian. Here $\star$   
 is the Hodge $\star$-operator. 
 Then it is established in \cite[Theorem 3.1]{Gaillard3} and  \cite[Corollary 2.2]{Pedon-these} that 
$$
 \mathbb{D}(G/K;  \tau) = \begin{cases}\mathbb{C}\left[ d d^* ,  d^* d\right] & \text { if } \tau=\tau_p, \, p<\frac{n-1}{2}\\ 
\mathbb{C}\left[d d^* , \star d \right] & \text { if } \tau=\tau_p,\; p=\frac{n-1}{2}, \\ 
\mathbb{C}[\Delta] & \text { if } \tau=\tau_p^\pm,\;  p=\frac{n}{2},\\
\mathbb{C}[\star, d\star d] & \text { if } \tau=\tau_p,\;  p=\frac{n}{2},
\end{cases}
$$
In particular,  $\mathbb D(G/K;\tau)$ is a commutative algebra except for $\tau=\tau_{n/2}$ (last case). \\

Above,  we have identified  the operators  $\Delta$, $d^*$, $d$ and $\star$ with their counterpart under the identification 
 $C^\infty(\bigwedge^p( H^n(\mathbb R)) \simeq C^\infty(G/K; \tau_p)$. This convention will be used throughout the paper.

 Let  $\tau_p\in \{\tau_1,\cdots,\tau_{\frac{n-1}{2}},   \tau_{\frac{n}{2}}\}$ and consider the component $\sigma_p$ of ${\tau_p}_{|M}$. 
 For $\lambda \in\mathbb C$, 
 we consider the following space of coclosed $p$-eigenforms of the Hodge-de Rham Laplacian:
$$
\mathcal{E}_{\sigma_p,\lambda}(G/K; \tau_p)=
	\{F\in C^\infty(G/K; \tau_p) : \; \Delta\, F=(\lambda^2+(\rho-p)^2)F, \,\text{ and }\;  d^* F=0\}.
$$

Then Gaillard's result may be stated as follows:
 
 \begin{theorem}[{\cite[Theorem 2']{Gaillard}}]\label{gaillard} 
Let $\tau_p=\tau_1,\cdots,\tau_{\frac{n-1}{2}}, \tau_{\frac{n}{2}}$  and 
$\lambda\in \mathbb{C}$. 
The Poisson transform 
$$
\mathcal{P}^{\tau_p}_{\sigma_p,\lambda} \colon C^{-\omega}(K/M;\sigma_p) \to \mathcal{E}_{\sigma_{p},\lambda}(G/K; \tau_p)
$$
is a  $G$-isomorphism  if and only if 
\begin{equation}\label{Pbijectivity}
i\lambda\notin \mathbb{Z}_{\leqslant 0}-\rho \cup \{p-\rho\}.
\end{equation}
\end{theorem}

 A similar result was announced in \cite[Theorem 4]{Juhl}.

To get the analog to Theorem \ref{gaillard} for the   pairs  $(\tau_p,\sigma_{p-1})$ with $1\leqslant p\leqslant \frac{n-1}{2}$, and $(\tau_p,\sigma_{p}^\pm)$ with $p=\frac{n-1}{2}$  , we  consider the following eigensapces:
$$
\mathcal{E}_{\sigma_{p-1},\lambda}(G/K; \tau_p)
=\{F\in C^\infty(G/K; \tau_p) : \;  \Delta F= (\lambda^2+(\rho-p+1)^2)F \; \text{and}\; d F=0\}, \; 1\leqslant p\leqslant \frac{n-1}{2},
$$
and 
$$
\mathcal{E}_{\sigma^\pm_p,\lambda}(G/K; \tau_p)
=\{F\in C^\infty(G/K; \tau_p) : \;  \star dF=\pm i^{p^2-1} \lambda F  \; \text{and}\;  d^*F=0\},\; p=\frac{n-1}{2}.
$$

\begin{corollary}\label{bijective1}
Let $\lambda\in \mathbb{C}$. 

$(1)$ For $p\leqslant\frac{n-1}{2}$,  the Poisson transform 
$$
\mathcal{P}^{\tau_p}_{\sigma_{p-1},\lambda}:C^{-\omega}(K/M;\sigma_{p-1})\to \mathcal{E}_{\sigma_{p-1},\lambda}(G/K;\tau_p)
$$
is a  $G$-isomorphism
if and only if 
\begin{equation}
i\lambda\notin \mathbb{Z}_{\leqslant 0}-\rho\cup \{\rho-p+1\}.	
\end{equation}

$(2)$ For $p=\frac{n-1}{2}$,  the Poisson transform 
$$
\mathcal{P}^{\tau_p}_{\sigma_{p}^\pm,\lambda}:C^{-\omega}(K/M;\sigma_{p}^\pm)\to \mathcal{E}_{\sigma_{p}^\pm,\lambda}(G/K;\tau_p)
$$
is a  $G$-isomorphism 
if and only if $$i\lambda\notin \mathbb Z_{\leqslant 0}-\rho.$$
\end{corollary}
 \begin{proof}
	Notice that the Hodge operator induces the equivalences $\tau_p\sim \tau_{n-p}$ and $\sigma_q\sim \sigma_{n-1-q}$. 
	
	(1) Using the identity  
$$
\star (\mathcal{P}_{p-1,\lambda}^{\tau_{p}}f)=(-1)^{p(p-1)}\mathcal{P}_{n-p,\lambda}^{\tau_{n-p}}(\star  f),
$$
 as well as the relations
$$ \begin{aligned}
 d^*\star&=(-1)^{n+1-p^2}\star\,d\\
   \Delta\star&=\star\Delta,
\end{aligned}$$
we easily see that the following diagram  
$$
\xymatrix{
    C^{-\omega}(K/M;\sigma_{p-1}) \ar[r]^{\mathcal{P}^{\tau_p}_{\sigma_{p-1},\lambda}} \ar[d]_\star  & \mathcal{E}_{\sigma_{p-1},\lambda}(G/K;\tau_p) \ar[d]^\star \\
    C^{-\omega}(K/M;\sigma_{n-p}) \ar[r]^{\mathcal{P}^{\tau_{n-p}}_{\sigma_{n-p},\lambda}} & \mathcal{E}_{\sigma_{n-p},\lambda}(G/K;\tau_{n-p})
  }
  $$
is commutative and the first part of the  corollary follows from Theorem \ref{gaillard}.\\
(2)  Let  $p=\frac{n-1}{2}$. Then $\sigma_p=\sigma^+_p\oplus \sigma^-_p$.  By using    \cite[(4.24) and (4.38)]{Pedon} we have
$$\operatorname{Im} \mathcal{P}_{\sigma^\pm_\frac{n-1}{2},\lambda}^{\tau_\frac{n-1}{2}}\subset  \mathcal{E}_{\sigma_{\frac{n-1}{2}}^\pm,\lambda}(G/K; \tau_{\frac{n-1}{2}}).
$$ 
Since $(\star d)^2=(-1)^{\frac{n(n-1)}{2}} d^* d$, it follows that $\mathcal{E}_{\sigma_p,\lambda}(G/K;\tau_p)=\mathcal{E}_{\sigma_p^+,\lambda}(G/K;\tau_p)\oplus \mathcal{E}_{\sigma_p^-,\lambda}(G/K;\tau_p)$. This together with Theorem \ref{gaillard} give the desired  result.
\end{proof}

 It remains to prove the analog of Theorem \ref{gaillard}   for the pair $(\tau_p^\pm,\sigma_p)$ with $p=\frac{n}{2}$.
 Since we could not obtain this result directly from Gaillard's theorem, we will instead turn to a more general result provided by Olbrich in \cite{olbrich}. Notice that we could have applied Olbrich's  result to the other cases, but we intentionally chose to remain within the framework of differential forms.

Let
$$
\mathcal{E}_{\sigma_p, \lambda}\left(G/K; \tau_p^{ \pm}\right)=\left\{F^{ \pm} \in C^{\infty}\left(G/K; \tau_p^{ \pm}\right): \Delta F^{ \pm}=\left(\lambda^2+\frac{1}{4}\right) F^{ \pm}\right\},\;\; p=\frac{n}{2}.
$$
Then the corollary below follows from \cite[Theorem 4.16]{olbrich}.
\begin{corollary}\label{cor-olbrich}
For $p=\frac{n}{2}$, the Poisson transform
$$
\mathcal{P}_{\sigma_p, \lambda}^{\tau_p^{ \pm}}: C^{-\omega}\left(K/M; \sigma_p\right) \rightarrow \mathcal{E}_{\sigma_p, \lambda}\left(G/K; \tau_p^{ \pm}\right)
$$
is a  $G$-isomorphism  if and only if
$$
i \lambda \notin \mathbb{Z}_{\leqslant 0}-\rho \cup\left\{-\frac{1}{2}\right\} .
$$	
\end{corollary}

It follows from Theorem \ref{gaillard}, Corollary \ref{bijective1} and Corollary \ref{cor-olbrich} that, for $1\leqslant p\leqslant \frac{n}{2}$ and $q=p-1, p$, the Poisson transform $\mathcal{P}_{ q,\lambda}^{ p}$ maps $L^r(K/M;\sigma_q)$ into  $\mathcal E_{\sigma_q,\lambda}(G/K ; \tau_p)$. We aim to precisely characterize the image of  $L^r(K/M;\sigma_q)$. This will be the focus of the following sections.

\section{Fatou-type theorem  and the Harish-Chandra $c$-function}\label{sec4} 
The following fact about the Jacobi functions will be necessary for subsequent sections (see, e.g.,  \cite{Ko}).
 The Jacobi function is defined as 
 \begin{equation}\label{jacobi}
 \phi_\mu^{(\alpha,\beta)}(t)={}_2F_1\left(\frac{i\mu+\alpha+\beta+1}{2},\frac{-i\mu+\alpha+\beta+1}{2};\alpha+1; -\sinh^2 t\right),  
 \end{equation}
 with $\Re(\alpha+1)>0$ and   ${}_2F_1$ is the classical hypergeometric function.
We shall need the following asymptotic behavior of   Jacobi functions. The asymptotic behavior of $ \phi_\mu^{(\alpha,\beta)}(t)$ as $t\rightarrow \infty$ is given by 
\begin{equation}\label{Jacobi}
\phi_\mu^{(\alpha,\beta)}(t)={\rm e}^{(i\mu-\alpha-\beta-1)t}\big(c_{\alpha,\beta}(\mu)+{\bf o}(1)\big)\,\; \text{as}\, \,  t\rightarrow \infty
\end{equation}
for \(\Re(i\mu)>0\), where 
\begin{equation}\label{simple}
  c_{\alpha,\beta}(\mu)=\frac{2^{\alpha+\beta+1-i\mu}\Gamma(\alpha+1)\Gamma(i\mu)}{\Gamma\left(\frac{i\mu+\alpha+\beta+1}{2}\right)\Gamma\left(\frac{i\mu+\alpha-\beta+1}{2}\right)}.
\end{equation}

Let $\tau\in\{\tau_1,\cdots,\tau_{\frac{n-1}{2}},   \tau^\pm_{\frac{n}{2}}\}$ and $\sigma\in \widehat{M}(\tau)$. We define the Hardy-type space $\mathcal{E}^{r}_{\sigma,\lambda}(G/K ; \tau)$,  as the  subspace of $F\in \mathcal E_{\sigma,\lambda}(G/K;\tau)$ such that
\begin{equation}\label{hardy}
\Vert F\Vert_{r,\lambda}:=\sup_{t>0}{\rm e}^{ (\rho-\Re(i\lambda))t}\left(\int_K \Vert F(ka_t) \Vert_{V_{\tau_p}}^r {\rm d}k\right)^\frac{1}{r} <\infty.
\end{equation}

\begin{proposition}\label{pro-gamma-lambda}   Let $\tau\in\{\tau_1,\cdots,\tau_{\frac{n-1}{2}},   \tau^\pm_{\frac{n}{2}}\}$ and $\sigma\in \widehat{M}(\tau)$.
For  $\lambda\in\mathbb{C}$ such that  $\Re(i\lambda)>0$, and for $r>1,$  the Poisson transform of every $f\in L^r(K/M;\sigma)$ belongs to $\mathcal{E}^{r}_{\sigma,\lambda}(G/K ; \tau).$ More precisely,
\begin{equation}\label{E1}
\Vert  \mathcal P_{\sigma,\lambda}^\tau f \Vert_{r,\lambda}  \leqslant \gamma_\lambda d_{\tau, \sigma}  \|f\|_{L^r\left(K / M ; \sigma\right)} ,
\end{equation}
for some positive constant $\gamma_\lambda$,  where $d_{\tau,\sigma}$ is as in \eqref{d26sept}. 
\end{proposition}

\begin{proof} We prove the inequality \eqref{E1}  for $\tau=\tau_p$, with   generic $p$,  and $\sigma=\sigma_q$ with $q=p, p-1$. We will  use the notation $\mathcal P_{\sigma_q,\lambda}^{\tau_p}=\mathcal P_{q,\lambda}^p$  and $d_{p,q}=d_{\tau_p,\sigma_q}$.  The  special  cases $p=\frac{n-1}{2}$ and $p=\frac{n}{2}$ can be similarly established.

By \eqref{pois} we have
\begin{equation*}\begin{split}
&\parallel \mathcal{P}_{q,\lambda}^pf(ka_t)\parallel_{ \bigwedge^p\mathbb C^n} \\
&\leqslant d_{p,q} \int_K {\rm e}^{-(\Re(i\lambda)+\rho)H(a_t^{-1}k^{-1}h)}\parallel \tau_p(\kappa(a_t^{-1}k^{-1}h)  \boldsymbol  \ib^{p}_{q} (f(h))\parallel_{ \bigwedge^p\mathbb C^n}{\rm d}h\\
& \leqslant d_{p,q} \int_K {\rm e}^{-(\Re(i\lambda)+\rho)H(a_t^{-1}k^{-1}h)}\parallel \boldsymbol  \ib^{p}_{q}(f(h))\parallel_{ \bigwedge^p\mathbb C^n}{\rm d}h,\\
\end{split}\end{equation*} 
where the last inequality follows from the unitarity of the representation $\tau_p$. 
 Since $\boldsymbol  \ib_q^p$ is an isometric embedding, it follows  that
\begin{eqnarray*} 
\parallel \mathcal{P}_{q,\lambda}^pf(ka_t)\parallel_{ \bigwedge^p\mathbb C^n} 
&\leqslant& d_{p,q} \int_K {\rm e}^{-(\Re(i\lambda)+\rho)H(a_t^{-1}k^{-1}h)}\parallel   f(h)\parallel_{\bigwedge^q\mathbb C^{n-1}}  {\rm d}h \\
&=& d_{p,q}\, e_{\lambda,t} \,\ast \parallel f(\cdot)\parallel_{\bigwedge^q\mathbb C^{n-1}}(k),
 \end{eqnarray*}
 where ${\rm e}_{\lambda,t}(g):={\rm e}^{-(\Re(i\lambda)+\rho)H(a_t^{-1}g^{-1})}$  and $\ast$ is the standard convolution  over $K$.
Therefore, for $r>1,$ Young's inequality implies \[
\left(\int_K \parallel \mathcal{P}^p_{q,\lambda}f(ka_t)\parallel_{ \bigwedge^p\mathbb C^n}^r{\rm d}k\right)^{{1}/{r}}\leqslant d_{p,q}
  \| {\rm e}_{\lambda,t}\|_{L^1(K)}
\parallel f\parallel_{L^r(K/M;\,\sigma_q)}.
\]
 \\
Now,  $$\parallel {\rm e}_{\lambda,t}\parallel_{L^1(K)}=\int_K {\rm e}^{-(\Re(i\lambda)+\rho)H(a_t^{-1}k^{-1})}{\rm d}k
  = \phi_{-i\Re(i\lambda)}^{(\rho-\frac{1}{2},-\frac{1}{2})}(t), $$ 
where $\phi_\mu^{(\alpha,\beta)}$ is the Jacobi function \autoref{jacobi}. 
Since $\Re(i\lambda)>0$, the asymptotic behavior \autoref{Jacobi} gives  
$$\parallel {\rm e}_{\lambda,t}\parallel_{L^1(K)}
=e^{(\Re(i\lambda)-\rho)t}\left(c_{\rho-\frac{1}{2},-\frac{1}{2}}(-i\Re(i\lambda))+ {\bf o}(1)\right) \; \text{as} \; t\to\infty,$$
where the constant $c_{\rho-\frac{1}{2},-\frac{1}{2}}(-i\Re(i\lambda))$ is given by \eqref{simple}. 
This finishes the proof of \eqref{E1} for generic $p.$   
\end{proof}

 Our next objective is to prove an analog of Proposition \ref{pro-gamma-lambda} where the inequality is reversed. 
We will establish a Fatou-type theorem for the Poisson transform to achieve this.

 Let $\bar{N}=\theta(N)$, where $\theta$ is the Cartan involution of $G$. For \(\lambda\in{\C}\),    the generalized Harish-Chandra $c$-function is defined  as
\begin{equation}\label{c-function}
\mathbf c(\lambda,\tau_p)=\int_{\bar{N}}{\rm e}^{-(i\lambda+\rho)H(\overline{n})}\tau_p(\kappa(\overline{n})){\rm d}\overline{n}\in \mathrm{End}(V_{\tau_p}).
\end{equation}
Here ${\rm d}\bar n$ is the Haar measure on $\bar N$ with the normalization
$$\int_{\bar N} {\rm e}^{-2\rho(H(\bar n))} {\rm d}\bar n=1.$$
The integral \autoref{c-function} converges for \(\lambda\) such that \(\Re(i\lambda)>0\) and has a meromorphic continuation to \({\C}\)   (see, e.g.   \cite{wallach2}).

Let $\tau_p\in\{\tau_1,\cdots,\tau_{\frac{n-1}{2}},   \tau^\pm_{\frac{n}{2}}\}$ and  let $\sigma_q\in \widehat M(\tau_p)$.  
 Since the restriction $\mathbf c(\lambda,\tau_p)_{|V_{\sigma_q}}$ commutes with $\sigma_q $,  Schur's lemma implies that  there exists a complex scalar $c_{\sigma_q}(\lambda,\tau_p)$ such that 
 $\mathbf c(\lambda,\tau_p)_{|V_{\sigma_q}}=c_{\sigma_q}(\lambda,\tau_p)  \mathrm{Id}_{V_{\sigma_q}}$. Therefore, for  generic $p$, we have
 \begin{equation}\label{decomp-c-1} 
 \mathbf c(\lambda,\tau_p)=c_{\sigma_{p-1}}(\lambda,\tau_p)  \mathrm{Id}_{\bigwedge^{p-1}\mathbb C^{n-1}}+c_{\sigma_p}(\lambda,\tau_p)\mathrm{Id}_{\bigwedge^{p}\mathbb C^{n-1}}   .
 \end{equation}
For simplicity,  we shall denote the  scalar component  $c_{\sigma_q}(\lambda,\tau_p)$ by $c_q(\lambda,p).$ 

Similarly, for  $p=\frac{n-1}{2}$ (when $n$ is odd), there exist three scalar coefficients $c_{{{n-3}\over 2}}(\lambda, {{n-1}\over 2}),$ 
$c_{{{n-1}\over 2}}^+(\lambda, {{n-1}\over 2})$ and $c_{      {{n-1}\over 2}}^-(\lambda, {{n-1}\over 2})$ such that 
 \begin{multline}
 \label{decomp-c-2-1} \mathbf c(\lambda,\tau_{{n-1}\over 2})= 
c_{ {{n-3}\over 2}}(\lambda, {{n-1}\over 2})  \mathrm{Id}_{\bigwedge^{{n-3}\over  2}\mathbb C^{n-1}}+\\
 c_{ {{n-1}\over 2}}^+(\lambda, {{n-1}\over 2})\mathrm{Id}_{\bigwedge_+^{{n-1}\over 2}\mathbb C^{n-1}}
 +c_{       {{n-1}\over 2}}^-(\lambda,\ {{n-1}\over 2})\mathrm{Id}_{\bigwedge_-^{{{n-1}\over 2}}\mathbb C^{n-1}}   
\end{multline}
 and for $p=\frac{n}{2}$ (when $n$ is even), we have 
 \begin{equation}\label{decomp-c-3-1}
   \mathbf c (\lambda,\tau_{n\over 2}^\pm)=c_{ {n\over 2}}^\pm (\lambda,  {n\over 2})\mathrm{Id}_{\bigwedge_\pm^{n\over 2}\mathbb C^{n}} . \end{equation}

To simplify our notation, we will use the unified symbol $c_q(\lambda, p)$ to represent all scalar  components appearing in the identities  \eqref{decomp-c-1}, \eqref{decomp-c-2-1}, and \eqref{decomp-c-3-1}. Although the same symbol will be employed, the value and interpretation of each constant will be distinguishable based on the specific context in which it is used.

 For generic $p$, explicit expressions of the scalars $c_{q}(\lambda,\ p)$    are provided in \cite{van-th} through a direct computation of the integral \autoref{c-function}.
Using a different approach,  we will calcule the scalar components   $c_{q}(\lambda,\ p)$  in both generic and special cases  (see    Proposition \ref{c-explicit} and Proposition \ref{c-explicit-special}). 

The following lemma is needed for later use.

\begin{lemma}\label{lemma-c-f} Let $\tau_p\in\{\tau_1,\cdots,\tau_{\frac{n-1}{2}},   \tau^\pm_{\frac{n}{2}}\}$ and $\sigma_q\in \widehat{M}(\tau_p)$.
 \begin{enumerate}[\upshape (1)]
\item For every $v\in V_{\sigma_q}$,
 \begin{equation}\label{lemma-c-f-1}
 \| \mathbf c(\lambda,\tau_p) \boldsymbol   \ib^p_q(v)\|_{V_{\tau_p}}=
|c_{ q}(\lambda, p)| \|v\|_{V_{\sigma_q}} .
\end{equation}
\item 
For every linear operator $L$ from a vector space $V$   to $V_{\sigma}$,
  \begin{equation}\label{lemma-c-f-2}  
 \| \mathbf c(\lambda,\tau_p) \boldsymbol   \ib^p_q L\|_{\rm HS}=  |c_{ q}(\lambda, p)| \|L\|_{\rm HS},
  \end{equation}
  where $\|\cdot\|_{\rm HS}$ is the Hilbert-Schmidt norm.
  \end{enumerate}
 \end{lemma}
\begin{proof} Follows immediately  from \eqref{decomp-c-1}, \eqref{decomp-c-2-1} and \eqref{decomp-c-3-1}.
\end{proof}

 \begin{proposition}[Fatou-type theorem] \label{Fatou}
 Let $\tau_p\in\{\tau_1,\cdots,\tau_{\frac{n-1}{2}},   \tau^\pm_{\frac{n}{2}}\}$ and $\sigma_q\in \widehat{M}(\tau_p)$.
Let \(\lambda\in {\C}\) such that \(\Re(i\lambda)>0\).   Then
\[\lim_{t\rightarrow \infty}{\rm e}^{(\rho-i\lambda)t} \mathcal{P}_{q,\lambda}^pf(ka_t)= \,d_{p,q} \mathbf c(\lambda,\tau_p)  \boldsymbol   \ib^p_q  (f(k)),\]

$(i)$ uniformly on $K$ for \(f\in C^\infty(K/M;\sigma_q)\),

$(ii)$ in the $L^r(K;\Lambda^p\mathbb C^n)$-sense,  for every $f\in$ \(L^r(K/M;\sigma_q)\).
\end{proposition}

\begin{proof}
Statement $(i)$   has been proved previously; see, for instance,   \cite{van} and  \cite{Yan}. \\
 $(ii)$  Assume $f\in L^r(K/M;\sigma_q)$ and  $\varepsilon>0$. By density, we can find a  $K$-finite vector $\varphi$ in $C^\infty(K/M;\sigma_q)$  such that 
 $\|f-\varphi\|_{L^r(K/M; \,\sigma_q)}<\varepsilon$. Put $\mathcal P^t_\lambda(f)(k)=\mathcal{P}_{q,\lambda}^p f(ka_t)$, 
  then 
  \begin{eqnarray*}
  \|    {\rm e}^{-(i\lambda-\rho)t} \mathcal P^t_\lambda(f)(k) -d_{p,q}\mathbf c(\lambda,\tau_p) \boldsymbol \ib_q^p f(k)\|_{\bigwedge^p\mathbb C^n}^r
  &  \hspace{-.2cm} \leqslant& \|   {\rm e}^{-(i\lambda-\rho)t} \mathcal P^t_\lambda(f-\varphi)(k)\|_{\bigwedge^p\mathbb C^n}^r  \\
  & & \hspace{-.7cm} +  \|   {\rm e}^{-(i\lambda-\rho)t} \mathcal P^t_\lambda(\varphi)(k)-d_{p,q} \mathbf c(\lambda,\tau_p)\boldsymbol  \ib_q^p\varphi(k) \|_{\bigwedge^p\mathbb C^n}^r\\
  & & \hspace{-.7cm} + d_{p,q}^r  \|\mathbf c(\lambda,\tau_p)\boldsymbol \ib_q^p \varphi(k)-\mathbf c(\lambda,\tau_p)\boldsymbol \ib_q^p f(k)\|_{\bigwedge^p\mathbb C^n}^r.
  \end{eqnarray*}
  From  Proposition \ref{pro-gamma-lambda} we obtain
  $$\int_K\|   {\rm e}^{-(i\lambda-\rho)t} \mathcal P^t_\lambda(f-\varphi)(k)\|_{\bigwedge^p\mathbb C^n}^r{\rm d}k 
  \leqslant \gamma_\lambda^r  d_{p,q}^r  \| f-\varphi\|^r_{L^r(K/M;\,\sigma_q)},$$
  and form part $(i)$  above it follows that
  $$\lim_{t\to \infty}  \int_K \|   {\rm e}^{-(i\lambda-\rho)t} \mathcal P^t_\lambda(\varphi)(k)-d_{p,q} \mathbf c(\lambda,\tau_p) \boldsymbol \ib_q^p \varphi(k) \|_{\bigwedge^p\mathbb{C}^n}{\rm d}k=0.$$
Further, according to     \autoref{lemma-c-f-1}  we obtain
  \begin{eqnarray*}
  \int_K  \|\mathbf c(\lambda,\tau_p)\boldsymbol \ib_q^p \varphi(k)-\mathbf c(\lambda,\tau_p)\boldsymbol \ib_q^p f(k)\|_{\bigwedge^p\mathbb C^n}^r 
  {\rm d}k
  &\leqslant & |c_q(\lambda, p)|^r \| f-\varphi\|^r_{L^r(K/M;\,\sigma_q)}.\\
   \end{eqnarray*}
In conclusion, we have
  $$\lim_{t\to \infty} \int_K \|   {\rm e}^{-(i\lambda-\rho)t} \mathcal P^t_\lambda(f)(k) -d_{p,q} \mathbf c(\lambda,\tau_p) \boldsymbol \ib_q^p f(k)\|^r_{\bigwedge^p\mathbb C^n} {\rm d}k \leqslant \varepsilon^r d_{p,q}^r(\gamma^r_\lambda + | c_q(\lambda,p)|^{r}),$$
  and this proves the desired statement.
\end{proof}

The following inequalities are crucial.
\begin{proposition}\label{nec-cdt} Let $\tau_p\in\{\tau_1,\cdots,\tau_{\frac{n-1}{2}},   \tau^\pm_{\frac{n}{2}}\}$ and $\sigma_q\in \widehat{M}(\tau_p)$.
For every  $\lambda\in \mathbb C$  such that   $\Re(i\lambda) >0$, there exists a positive constant $\gamma_\lambda$ such that for all  $f\in L^r (K/M;\sigma_q)$,   $1<r<\infty$, we have
\begin{equation}\label{esti-F}
d_{p,q} |c_q(\lambda,p)| \| f\|_{ L^r (K/M;\,\sigma_q)}\leqslant \|\mathcal{P}_{q,\lambda}^p f\|_{ {r,\lambda}} \leqslant  d_{p,q}\,\gamma_\lambda \| f\|_{ L^r (K/M;\,\sigma_q)}.
\end{equation}
\end{proposition}
\begin{proof} The right-hand side inequality is nothing but the estimate   \autoref{E1}.  For the left-hand side inequality,  
by Proposition \ref{Fatou}[$(ii)$], there exists a sequence  $(t_j)_j$ with $t_j\to\infty$ such that 
$$\lim_{j\rightarrow \infty} \|{\rm e}^{(\rho-i\lambda)t_j} \mathcal{P}_{q,\lambda}^pf(ka_{t_j})\|_{\bigwedge^p\mathbb C^n}=\|d_{p,q} \,\mathbf c(\lambda,\tau_p) \boldsymbol \ib_q^p (f(k))\|_{\bigwedge^p\mathbb C^n} $$
almost everywhere in $K$. Consequently, by  the classical Fatou theorem and  \autoref{lemma-c-f-1} we get
 
\begin{eqnarray*}
d_{p,q}^r |c_q(\lambda,p)|^r \int_K \| f(k)\|^r_{\bigwedge^q\mathbb C^{n-1}}{\rm d}k
&\leqslant&  \sup_{j} {\rm e}^{r\Re(\rho -i\lambda)t_j} \int_K \|\mathcal P^{t_j}_\lambda (f)(k)\|^r_{\bigwedge^p\mathbb C^n}{\rm d}k,
\end{eqnarray*}
which  implies 
$$d_{p,q}  |c_q(\lambda,p)|  \, \|f\|_{ L^r (K/M;\sigma_q)} \leqslant  \|\mathcal{P}_{q,\lambda}^p f\|_{ {r,\lambda}}.$$

\end{proof}

 In the rest of this section, we will see how the asymptotic behavior formula given in Proposition \ref{Fatou} will allow us to give explicitly all the  scalar components  appearing in \eqref{decomp-c-1}, \eqref{decomp-c-2-1}, and \eqref{decomp-c-3-1}.

Let us recall that a continuous function \(F\colon G\rightarrow \mathrm{End}(V_{\tau})\) is called elementary \(\tau_p\)-spherical if \(F\) satisfies:
\begin{itemize}
\item[$(i)$]  $F$ is a $\tau_p$-radial function, i.e.  $F(k_1gk_2)=\tau_p(k_2)^{-1}F(g)\tau_p(k_1^{-1}),$  for all $g\in G$ and $k_1,k_2\in K.$
\item[$(ii)$] For all $v\in V_{\tau_p}$, $g \mapsto  F(g)v$  is a  joint-eigenfunction of all $D\in \mathbb {D}(G/K ; \tau_p)$ with $F(e)=\text{Id}.$
\end{itemize}
It is known that a  $\tau_p$-radial function $F \colon G \to \mathrm{End}(V_{\tau_p})$   is determined by its restriction $F_{|_A}$ to the subgroup $A$ of $G$. Since $A$ and $M$ commute, $F_{|_A}$ becomes an $M$-morphism of $V_{\tau}$. Furthermore,   ${\tau_p}_{|_M}$ decomposes with multiplicity one, therefore by Schur's lemma, $F_{|_A}$ is scalar on each $M$-irreducible component  $V_{{\sigma_q}}$, for $\sigma_q\in\widehat{M}(\tau_p)$. Thus
$$F_{|_A}(a_t)=\sum_{\sigma_q\in\widehat{M}(\tau_p)} f_{\sigma_q}(t) \mathrm{Id}_{ V_{\sigma_q}},
$$ 
the coefficients 
 $f_{\sigma_q}(t)$  are called the scalar components of $F$.
 
It is also known that any $\tau_p$-spherical function is given by the following  Eisenstein integral   by
\begin{equation}\label{spherical}
\Phi_{\sigma_q, \lambda}^{\tau_p}(g) = d_{p,q}^2\int_K {\rm e}^{-(i\lambda+\rho)H(g^{-1}k)}\tau_p(\kappa(g^{-1}k))\boldsymbol   \ib^p_q(\boldsymbol \pi_p^q(\tau_p(k)^{-1})){\rm d}k,
\end{equation}
 where $\boldsymbol \pi_p^q$ is the dual endomorphism of $\boldsymbol   \ib^p_q$. By \cite{Pedon-these}, the scalar components of $\Phi_{\sigma_q,\lambda}^{\tau_p}$ are given in terms of the Jacobi function $\phi_\mu^{(\alpha, \beta)}$ as follows:
  \begin{enumerate}[\upshape (1)] 
 \item When $p$ is generic:

\begin{enumerate}[\upshape (a)]

 \item The scalar components $\varphi_{p-1,\lambda}, \varphi_{p,\lambda}$ of $\Phi^{\tau_p}_{\sigma_p,\lambda}$ are given by
 	\begin{align}
 	\varphi_{p-1,\lambda}(t)&=\phi_\lambda^{(\frac{n}{2},-\frac{1}{2})}(t), \label{phi1}\\
 	\varphi_{p,\lambda}(t)&=\frac{n}{n-p}\phi_\lambda^{(\frac{n}{2}-1,-\frac{1}{2})}(t)-\frac{p}{n-p}\cosh(t)\phi_\lambda^{(\frac{n}{2},-\frac{1}{2})}(t).\label{phi2}
 	\end{align}	
 
 \item 
The scalar components $\varphi_{p-1,\lambda}, \varphi_{p,\lambda}$ of $\Phi^{\tau_p}_{\sigma_{p-1},\lambda}$ are given by
\begin{align} 
 \varphi_{p-1,\lambda}(t)&=\frac{n}{p}\phi_\lambda^{(\frac{n}{2}-1,-\frac{1}{2})}(t)-\frac{n-p}{p}\cosh(t)\phi_\lambda^{(\frac{n}{2},-\frac{1}{2})}(t),\label{phi3}\\
 \varphi_{p,\lambda}(t)&=\phi_\lambda^{(\frac{n}{2},-\frac{1}{2})}(t).\label{phi4}
\end{align}
  \end{enumerate}

  \item In the special case $p=\frac{n-1}{2}:$ 
\begin{enumerate} [\upshape (a)]
\item 
 The scalar components $\varphi_{p-1,\lambda}, \varphi^+_{p,\lambda}, \varphi^-_{p,\lambda} $ of $\Phi^{\tau_p}_{\sigma_{p-1},\lambda}$ are given by
\begin{align} 
\varphi_{p-1,\lambda}(t)&
=\frac{2n}{n-1}\phi_\lambda^{(\frac{n}{2}-1,-\frac{1}{2})}(t)-\frac{n+1}{n-1}\cosh(t)\phi_\lambda^{(\frac{n}{2},-\frac{1}{2})}(t),\label{phi1-sp1}\\
\varphi^\pm_{p,\lambda}(t)&=\phi_\lambda^{(\frac{n}{2},-\frac{1}{2})}(t).\label{phi2-sp1}
 \end{align} 
\item The scalar components $\varphi_{p-1,\lambda}, \varphi^+_{p,\lambda}, \varphi^-_{p,\lambda} $ of $\Phi^{\tau_p}_{\sigma_p^\pm,\lambda}$    are given by
 \begin{align} 
 \varphi_{p-1,\lambda}(t)&=\phi_\lambda^{(\frac{n}{2},-\frac{1}{2})}(t), \label{phi3-sp1}\\
 \varphi_{p,\lambda,}^+(t)&=
 \frac{2n}{n+1}\phi_\lambda^{(\frac{n}{2}-1,-\frac{1}{2})}(t)
 -\frac{n-1}{n+1}\cosh(t)\phi_\lambda^{(n/2,-1/2)}(t)
 \pm\frac{i2\lambda}{n+1}\sinh(t) \phi_\lambda^{(\frac{n}{2},-\frac{1}{2})}(t),  \label{phi4-sp1}\\
\varphi_{p,\lambda}^-(t)&=
\frac{2n}{n+1}\phi_\lambda^{(\frac{n}{2}-1,-\frac{1}{2})}(t)
-\frac{n-1}{n+1}\cosh(t)\phi_\lambda^{(\frac{n}{2},-\frac{1}{2})}(t)\mp\frac{i2\lambda}{n+1}\sinh(t)  \phi_\lambda^{(\frac{n}{2},-\frac{1}{2})}(t). \label{phi5-sp1}
 \end{align}

 \end{enumerate}

\item In the special case $p=\frac{n}{2}$,   the scalar component $\varphi_{p,\lambda}^\pm$ of $\Phi^{\tau_p^\pm}_{\sigma_p}$ is given by
 \begin{equation}\label{phi-sp2}
 	\varphi_{p,\lambda}^+(t)=\varphi_{p,\lambda}^-(t)=\cosh\left(\frac{t}{2}\right) \phi_{2\lambda}^{(\frac{n}{2}-1,\frac{n}{2}+1)}\left(\frac{t}{2}\right).
 \end{equation}

    \end{enumerate}
 
  Below we will give explicitly the scalar components of the generalized Harish-Chandra $c$-function  appearing in \eqref{decomp-c-1}.

\begin{proposition}\label{c-explicit}
Let \(\lambda\in\mathbb C\) such that \(\Re(i\lambda)>0\). For generic $p$, the generalized Harish-Chandra $c$-function is given by 
\begin{equation*}
\mathbf c(\lambda,\tau_p)=c_{p-1}(\lambda,p)  \mathrm{Id}_{\bigwedge^{p-1}\mathbb C^{n-1}}+c_p(\lambda,p)\mathrm{Id}_{\bigwedge^{p}\mathbb C^{n-1}},
\end{equation*}
where the scalar coefficients are explicitly given by
$$
c_{p-1}(\lambda,p)    
 =\frac{i\lambda- \rho+p-1}{i\lambda+\rho} c(\lambda),\qquad 
c_p(\lambda,p)   
 =\frac{i\lambda+ \rho- p}{i\lambda+\rho} c(\lambda),
$$
with  $$c(\lambda)=c_{\frac{n}{2}-1,-\frac{1}{2}}(\lambda)= 2^{\rho-i\lambda} \frac{\Gamma(i\lambda) \Gamma\left(\rho+\frac{1}{2}\right)}
{\Gamma\left(\frac{i\lambda +\rho}{2} \right) \Gamma\left(\frac{i\lambda +\rho+1}{2} \right)}.$$
\end{proposition}

\begin{proof} Assume that $p$ is generic and $q=p-1, p$. For  \(\lambda\in\mathbb C\) such that \(\Re(i\lambda)>0\), we have 
$$   \Phi^p_{q, \lambda}( ka_t)= d_{p,q} \mathcal P_{q,\lambda}^p\left( \boldsymbol\pi_p^q(\tau_p(k^{-1}))\right)(a_t).$$
  Proposition \ref{Fatou}  implies
\begin{equation}\label{sph1}
  \Phi^p_{q, \lambda}( a_t)=d^2_{p,q} \mathbf c(\lambda,\tau_p){\rm e}^{(i\lambda-\rho)t}\left( \boldsymbol\pi^q_p+{\bf o}(1)\right)\; 
  \text{ as } t\rightarrow \infty,
\end{equation} 
with
$$
d_{p,q}^2= 
 \begin{cases}
 {\frac{n}{n-p}} & \text{if}\; q=p,\\
 {\frac{n}{p}} & \text{if}\; q=p-1.
\end{cases}
$$
Let us first consider the case \(q=p\). 
Using the asymptotic behavior \autoref{Jacobi} of   Jacobi functions  together with the relation 
$$c_{\frac{n}{2},-\frac{1}{2}}(\lambda)=\frac{2n}{i\lambda+\rho}c_{\frac{n}{2}-1,-\frac{1}{2}}(\lambda),$$ 
 we obtain
\begin{eqnarray*} 
 \varphi_{p,p}(\lambda,t)&  {=}& \frac{1}{n-p}{\rm e}^{(i\lambda-\rho)t}
 \left( n c_{\frac{n}{2}-1,-\frac{1}{2}}(\lambda)-\frac{p}{2} c_{\frac{n}{2},-\frac{1}{2}}(\lambda)+{\bf o}(1)\right)    \text{ as } t\rightarrow \infty \\
 & {=}&{\rm e}^{(i\lambda-\rho)t}\frac{n}{n-p}c_{\frac{n}{2}-1,-\frac{1}{2}}(\lambda)\left(\frac{i\lambda+\rho-p}{i\lambda+\rho}+ {\bf o}(1)\right)    \text{ as } t\rightarrow \infty. 
\end{eqnarray*}
Similarly, we get 
 $$   \varphi_{p,p-1}(\lambda,t) {=} {\rm e}^{(i\lambda-\rho-1)t} \left(c_{\frac{n}{2},-\frac{1}{2}}(\lambda)+ {\bf o}(1)\right)   \text{ as } t\rightarrow \infty.$$ 
Thus 
 \begin{multline*} 
  \Phi^{\tau_p}_{\sigma_p,\lambda}( a_t)={\rm e}^{(i\lambda-\rho-1)t} \left(c_{\frac{n}{2},-\frac{1}{2}}(\lambda)+ {\bf o}(1)\right)  \mathrm{Id}_{\bigwedge^{p-1}\mathbb C^{n-1}}\\
 +{\rm e}^{(i\lambda-\rho)t}\frac{n}{n-p}c_{\frac{n}{2}-1,-\frac{1}{2}}(\lambda)\left(\frac{i\lambda+\rho-p}{i\lambda+\rho}+ {\bf o}(1)\right) \mathrm{Id}_{\bigwedge^{p}\mathbb C^{n-1}} ,
  \end{multline*}
from which we deduce that 
\begin{equation}\label{sph2}
\lim_{t\rightarrow \infty} {\rm e}^{(\rho-i\lambda)t}   \Phi^{\tau_p}_{\sigma_p,\lambda}(a_t)=
\frac{n}{n-p} \left( \frac{i\lambda+\rho-p}{i\lambda+\rho}\right) c_{\frac{n}{2}-1,-\frac{1}{2}}(\lambda)\mathrm{Id}_{\bigwedge^{p}\mathbb C^{n-1}}.
\end{equation}
Finally, by identification of \autoref{sph1} and \autoref{sph2} it follows that  
 $$ c_p(\lambda,p)= \frac{i\lambda+\rho-p}{i\lambda+\rho}c_{\frac{n}{2}-1,-\frac{1}{2}}(\lambda)= \frac{i\lambda+\rho-p}{i\lambda+\rho}c(\lambda).$$
Similarly, for $q=p-1$ we can prove that 
\[\lim_{t\rightarrow \infty} {\rm e}^{(i\lambda-\rho)t}   \Phi^{\tau_p}_{\sigma_{p-1},\lambda}(a_t)=\frac{n}{p} 
\left( \frac{i\lambda-\rho+p-1}{i\lambda+\rho}\right) c_{\frac{n}{2}-1,-\frac{1}{2}}(\lambda)  \mathrm{Id}_{\bigwedge^{p-1}\mathbb C^{n-1}},\] 
from which we deduce that
$$c_{p-1}(\lambda,p)=\frac{i\lambda-\rho+p-1}{i\lambda+\rho}c_{\frac{n}{2}-1,-\frac{1}{2}}(\lambda) = \frac{i\lambda-\rho+p-1}{i\lambda+\rho} c(\lambda).$$
\end{proof}
 
 The scalar components   appearing in \eqref{decomp-c-2-1}  and \eqref{decomp-c-3-1} are given below.
\begin{proposition}\label{c-explicit-special}
Let \(\lambda\in\mathbb C\) such that \(\Re(i\lambda)>0\). In the special cases, the generalized Harish-Chandra $c$-function is given as follows:
  \begin{enumerate}[\upshape (1)]  
  \item For  $p=\frac{n-1}{2}$ (when $n$ is odd), we have 
 \begin{multline*}
 \label{decomp-c-2} \mathbf c(\lambda,\tau_{{n-1}\over 2})= 
c_{ {{n-3}\over 2}}(\lambda, {{n-1}\over 2})  \mathrm{Id}_{\bigwedge^{{n-3}\over  2}\mathbb C^{n-1}}+\\
 c_{ {{n-1}\over 2}}^+(\lambda, {{n-1}\over 2})\mathrm{Id}_{\bigwedge_+^{{n-1}\over 2}\mathbb C^{n-1}}
 +c_{       {{n-1}\over 2}}^-(\lambda,\ {{n-1}\over 2})\mathrm{Id}_{\bigwedge_-^{{{n-1}\over 2}}\mathbb C^{n-1}}   
\end{multline*}
where the scalar coefficients are explicitly given by  
\begin{equation*}
c_{{n-3}\over 2}(\lambda,{{n-1}\over 2})    
 =\frac{2n}{n-1}\frac{i\lambda- 1}{i\lambda+\rho} c(\lambda),
\qquad
c_{ {{n-1}\over 2}}^\pm(\lambda,{{n-1}\over 2})   
 =\frac{2n}{n+1}\frac{i\lambda}{i\lambda+\rho} c(\lambda),
\end{equation*}
with  $$c(\lambda)=c_{\frac{n}{2}-1,-\frac{1}{2}}(\lambda)= 2^{\rho-i\lambda} \frac{\Gamma(i\lambda) \Gamma\left(\rho+\frac{1}{2}\right)}
{\Gamma\left(\frac{i\lambda +\rho}{2} \right) \Gamma\left(\frac{i\lambda +\rho+1}{2} \right)}.$$

\item For $p=\frac{n}{2}$ (when $n$ is even), we have 
 \begin{equation*}\label{decomp-c-3}
   \mathbf c (\lambda,\tau_{n\over 2}^\pm)=c_{ {n\over 2}}^\pm (\lambda,  {n\over 2})\mathrm{Id}_{\bigwedge_\pm^{n\over 2}\mathbb C^{n}} . \end{equation*}
\end{enumerate}

where the scalar coefficient is given by
$$
c_{n\over 2}^\pm(\lambda,{n\over 2})=c_{\frac{n}{2}-1, \frac{n}{2}+1}(2 \lambda)
=
2^{n+1-2i \lambda} \frac{\Gamma\left(\frac{n}{2}\right) \Gamma(2i \lambda)}{\Gamma\left(\frac{2i \lambda+n+1}{2}\right) \Gamma\left(\frac{2i \lambda-1}{2}\right)} .
$$
\end{proposition}
\begin{proof}
	The proof is similar to the generic case.
\end{proof}

\section{The $L^2$-range of the Poisson transform}\label{sec5}
Recall that our primary objective is to describe the image of the space   $L^r(K/M;\sigma_q)$ under the Poisson transform $\mathcal{P}_{q,\lambda}^p$, for $1< r<\infty$. To achieve this, we will first examine the case when   $r=2$.

Let  $\tau_p\in \{\tau_1,\cdots,\tau_{\frac{n-1}{2}}, {\tau^\pm_{\frac{n}{2}}}\}$ and $\sigma_q\in\widehat{M}(\tau_p)$ with $q=p, p-1$. 
Let $(\delta,V_\delta)$ be an element in $\widehat{K}(\sigma_q)$, where $\widehat K(\sigma_q)\subset \widehat K$ 
denotes the subset of those representations that include   $\sigma_q$ when restricted to $K$. By the branching law, $\sigma_q$ occurs in the restriction $\delta_{\big |M}$ with multiplicity one and therefore $\dim \mathrm{Hom}_M(V_\delta, V_{\sigma_q})=1$.  We choose    the orthogonal projection $P_\delta : V_\delta\to V_{\sigma_q}$ as   a generator of $\mathrm{Hom}_M(V_\delta, V_{\sigma_q})$.

let   $(v_j)_{j=1}^{d_\delta}$ be an orthonormal basis for $V_\delta$, where $d_\delta=\dim V_\delta$. Then  the functions 
$$k\mapsto \phi^\delta_j(k)=P_\delta(\delta(k^{-1})v_j), \; \;\; 1\leqslant j\leqslant d_\delta, \,\, \delta\in \widehat{K}(\sigma)$$ 
define an orthogonal basis of the space $L^2( K/M; \sigma_q)$, see, e.g. \cite{wallach}. 
  Thus, the  Fourier expansion of any  $ f\in L^2( K/M; \sigma_q)$ is given by   
\[f(k)=\sum_{\delta\in\widehat{K}(\sigma_q)}\sum_{j=1}^{d_\delta} a^\delta_{j} \phi^{\delta}_j(k),\]
with 
\begin{equation}\label{L2-norm}
\displaystyle \parallel f\parallel^2_{L^2(K/M;\, \sigma_q)}=\sum_{\delta\in\widehat{K}(\sigma_q)} \frac{d_{\sigma_q}} {d_\delta} \sum_{j=1}^{d_\delta}\mid a^\delta_{j}\mid^2.
\end{equation}

\begin{theorem}\label{th-cas-L2} Let $\tau=\tau_p\in\{\tau_1,\cdots,\tau_{\frac{n-1}{2}},   \tau^\pm_{\frac{n}{2}}\}$ and $\sigma=\sigma_q\in \widehat{M}(\tau)$ with $q=p-1, p$.
   Assume
  $\lambda\in \mathbb C$ such that 
  \begin{equation*} 
\begin{cases}
\Re(i\lambda)>0&\text{ if $q=p$},\\ 
 \Re(i\lambda)>0 \, \text{ and } \; i\lambda\not=\rho-p+1& \text{ if $q=p-1$}. 
 \end{cases}
 \end{equation*}
 The Poisson transform $\mathcal{P}_{\sigma,\lambda}^{\tau}$ is  a topological isomorphism
 of the space $L^2\left(K/M; {\sigma}\right)$ onto the space $\mathcal{E}^2_{\sigma,\lambda}(G/K ; {\tau})$.
  Moreover,  for every $f\in L^2\left(K/M;{\sigma}\right)$, 
  there exists a positive constant $\gamma_\lambda$ such that  
\begin{equation}\label{esti-F-2}
  | c_\sigma(\lambda,\tau)| \;\| f\|_{L^r\left(K/M; {\sigma}\right)}\leqslant \sqrt{\frac{\dim \sigma}{\dim \tau}} \;
 \|\mathcal{P}_{\sigma,\lambda}^{\tau} f\|_{2,\lambda} \leqslant    \gamma_\lambda \| f\|_{ L^2\left(K/M; {\sigma}\right)},
\end{equation}
 where $ | c_\sigma(\lambda,\tau)|$ are  given explicitly in Proposition \ref{c-explicit} and Proposition \ref{c-explicit-special}.
\end{theorem}

\begin{proof} 
We will present  the proof only when $p$ is generic; the special cases $p=(n-1)/2$ and $p=n/2$ are analogous and are left to the reader.

Assume $p$ generic and $q=p-1$ or $q=p$. \ Recall the notation  $\mathcal P_{q,\lambda}^p=\mathcal P_{\sigma_q,\lambda}^{\tau_p}$. From Proposition \ref{nec-cdt} it follows that  the right-hand side of the estimate \eqref{esti-F-2} holds and that  $\mathcal{P}_{q,\lambda}^p$ is a  an injective continuous map from $L^2(K/M;\sigma_q)$    into $\mathcal{E}^{2}_{\sigma_q,\lambda}(G/K ; \tau_p)$.
 
 On the other hand, 
for $F\in \mathcal{E}^{2}_{\sigma_q,\lambda}(G/K ; \tau_p)$, by   Theorem \ref{gaillard},  Corollary \ref{bijective1} and Corollary \ref{cor-olbrich}, there exists a hyperform  $f\in C^{-\omega}(K/M;\sigma_q)$  such that 
$F=\mathcal{P}_{q,\lambda}^p f$. We can write $f$ as 
$$
f(k)=\sum_{\delta\in\widehat{K}(\sigma_q)}  \sum_{j=1}^{d_\delta}a^\delta_{j}  P_\delta(\delta(k^{-1})v_j),
$$
 then 
\begin{align*}
F(g)=d_{p,q}\sum_{\delta\in\widehat{K}(\sigma_q)}  \sum_{j=1}^{d_\delta}a^\delta_{j}\int_K {\rm e}^{-(i\lambda+\rho)H(g^{-1}k)}\tau_p(\kappa(g^{-1}k)) \boldsymbol  \ib_q^p P_\delta(\delta(k^{-1})v_j){\rm d}k,
\end{align*} 
where $d_{p,q}=\sqrt{\frac{\dim \tau_p}{\dim \sigma_q}}.$
Define $\Phi_{\lambda,\delta}$ by
\begin{equation}\label{Eisen}
\Phi_{\lambda,\delta}(g)(v)= d_{p,q} \int_K {\rm e}^{-(i\lambda+\rho)H(g^{-1}k)}\tau_p(\kappa(g^{-1}k))\boldsymbol  \ib_q^p P_\delta(\delta(k^{-1}) v){\rm d}k,  
\end{equation}
for $g\in G$ and $v\in V_\delta$.
Clearly  \(\displaystyle \Phi_{\lambda,\delta}(k_1gk_2)=\tau_p(k_2^{-1})\Phi_{\lambda,\delta}(g)\delta(k_1^{-1})\) for every  \(g\in G\) and \(k_1,k_2\in K\).
Further
\begin{align*}
\int_K\langle F(ka_t),F(ka_t)\rangle_{\bigwedge^p\mathbb C^{n}}{\rm d}k
=  \sum_{\delta,\delta' }   \sum_{j, \ell}a_{j}^\delta\,\overline{a^{\delta'}_{\ell}}\int_K \langle \Phi_{\lambda,\delta}(ka_t)v_j,\Phi_{\lambda,\delta'}(ka_t)v_\ell\rangle_{\bigwedge^p\mathbb C^{n}}\, {\rm d}k.
\end{align*}
By the covariance property and Schur's lemma, we obtain
\begin{align*}
\int_k \langle \Phi_{\lambda,\delta}(ka_t)v_j,\Phi_{\lambda,\delta'}(ka)v_\ell\rangle_{\bigwedge^p\mathbb C^{n}}\, {\rm d}k &=\int_K \langle \Phi_{\lambda,\delta'}(a_t)^\ast\Phi_{\lambda,\delta}(a_t)\delta(k^{-1})v_j,\delta'(k^{-1})v_\ell\rangle_{V_\delta}\,{\rm d}k\\
&=
\begin{cases}
0&\text{if $\delta'\nsim\delta$} \\
\frac{1}{d_\delta} \tr\left(\Phi_{\lambda,\delta}(a_t)^\ast\Phi_{\lambda,\delta}(a_t)\right)\langle v_j,v_\ell\rangle_{V_\delta} & \text{otherwise}
\end{cases}
\end{align*}
Thus
\begin{eqnarray*}
\int_K\parallel F(ka_t)\parallel_{\bigwedge^p\mathbb C^n}^2{\rm d}k
&=&   \sum_{\delta\in \widehat{K}(\sigma)}   \frac{1}{d_\delta}\sum_{j=1}^{d_\delta} |a_{j}^\delta|^2  \tr\left(\Phi_{\lambda,\delta}(a_t)^\ast\Phi_{\lambda,\delta}(a_t)\right),\\
&=  &\sum_{\delta } \frac{1}{d_\delta}   \|    \Phi_{\lambda,\delta}(a_t) \|_{\rm{HS}}^2 \sum_j   |a^\delta_{j}|^2,
\end{eqnarray*}
where $\|\cdot \|_{\rm{HS}}$ is the Hilbert-Schmidt norm. Hence, for a finite subset $\Omega\subset \widehat K(\sigma_q)$ we get
\begin{eqnarray*}
 \sum_{\delta\in \Omega}  \frac{1}{d_\delta} \sum_j   \|   a^\delta_{j} {\rm e}^{(\rho-i\lambda)t} \Phi_{\lambda,\delta}(a_t) \|_{\rm{HS}}^2
 & \leqslant &\displaystyle \sup_{t>0}{\rm e}^{2(\rho-\Re(i\lambda))t}\int_K\parallel F(ka_t)\parallel_{\bigwedge^p\mathbb C^n}^2{\rm d}k,\\
 &=&\|F\|^2_{ {2,\lambda}}.
 \end{eqnarray*}
Under the assumption $\Re(i\lambda)>0,$ we may use Proposition  \ref{Fatou}, {\it i.e.},
\begin{equation}\label{esti}
\lim_{t\rightarrow \infty}{\rm e}^{(\rho-i\lambda)t}\Phi_{\lambda,\delta}(a_t)= d_{p,q}\mathbf c(\lambda,\tau_p)\boldsymbol \ib_q^p P_\delta,
\end{equation}
and  \autoref{lemma-c-f-2} to obtain
$$  d_{p,q}^2|c_q(\lambda,p)|^2\sum_{\delta\in \Omega}  \frac{1}{d_\delta} 
 \sum_j    \|   a^\delta_{j} P_\delta   \|_{\rm{HS}}^2 \leqslant \|F\|^2_{ {2,\lambda}}.
 $$
 That is
 $$  d_{p,q}^2 |c_q(\lambda,p)|^2\sum_{\delta\in \Omega}  \frac{1}{d_\delta} 
 \sum_j   d_{\sigma_q}   |  a^\delta_{j}|^2   \leqslant \|F\|^2_{{2,\lambda}}.
 $$
Since the subset $\Omega\subset \widehat K(\sigma_q)$ is arbitrary, it follows that  
\begin{align*}
d_{p,q}^2 |c_q(\lambda,p)|^2 \sum_{\delta\in \widehat{K}(\sigma_q)}\frac{d_{\sigma_q}}{d_\delta}\sum_{j}\mid a_{j}^\delta\mid^2 \leqslant \parallel F\parallel^2_{ {2,\lambda}}<\infty.
\end{align*}
This shows that  \(f\in L^2(K/M;\sigma_q)\)  with    
$$ d_{p,q} |c_q(\lambda,p)| \parallel f\parallel_{L^2(K/M;\,\sigma_q)}\leqslant \parallel \mathcal{P}_{q,\lambda}^p f\parallel_{ {2,\lambda}}.$$  
\end{proof}

\begin{lemma}\label{lemma-lim-sum} 
For  any $\lambda\in\mathbb C$ such that $\Re(i\lambda)>0$ 
we have 
 \begin{equation*}
   \sup _{t>0} {\rm e}^{(\rho-\Re(i\lambda))t} \|\Phi_{\lambda,\delta} (a_t)\|_{\rm{HS}} \leqslant \gamma_\lambda d_{p,q}  \| P_\delta\|_{\rm{HS}}= \gamma_\lambda d_{p,q}  \sqrt{d_{\sigma_q}}.
   \end{equation*}
\end{lemma}
\begin{proof}
Let $v\in V_{\tau_p}$. By  Proposition \ref{pro-gamma-lambda} we have    $$\sup _{t>0} {\rm e}^{(\rho-\Re(i\lambda))t} 
   \left(
   \int_K \|\mathcal{P}_{q,\lambda}^p (P_\delta(\delta^{-1}(\cdot)v))( ka_t) \|^2_{V_{\tau_p}  } {\rm d}k\right)^{1/2}\leqslant \gamma_\lambda d_{p,q}    
   \|P_\delta(\delta^{-1}(\cdot) v)\|_{L^2(K/M;\,\sigma_q)}.$$
Since $\mathcal{P}_{q,\lambda}^p (P_\delta(\delta^{-1}(\cdot)v))( ka_t)=\Phi_{\lambda,\delta}(ka_t)(v)$, we get
   \begin{eqnarray*}
    \int_K \|\mathcal{P}_{q,\lambda}^p (P_\delta(\delta^{-1}(\cdot)v)( ka_t)) \|^2_{V_{\tau_p}  } {\rm d}k
    &=& \int_K \langle \Phi_{\lambda,\delta} (a_t) \delta(k^{-1}) v ,  \Phi_{\lambda,\delta} (a_t) \delta(k^{-1}) v \rangle_{V_{\tau_p}  }  \, {\rm d}k,\\
    &=&\frac{1}{d_\delta} \tr \left( \Phi_{\lambda,\delta} (a_t)^*  \Phi_{\lambda,\delta} (a_t)\right) \|v\|_{V_\delta}^2,\\
    &=&\frac{1}{d_\delta} \| \Phi_{\lambda,\delta} (a_t)\|^2_{\rm{HS}} \|v\|_{V_\delta}^2.\\
   \end{eqnarray*}
 Now the desired inequality follows from
   $$    \|P_\delta(\delta^{-1}(\cdot) v)\|^2_{L^2(K/M;\,\sigma_q)} =\frac{d_{\sigma_q}}{d_\delta} \, \| v\|^2_{V_\delta}.$$
 \end{proof}

\begin{lemma}\label{key-lemma}   
For any $\delta\in\widehat{K}(\sigma_q)$ and any $\lambda\in\mathbb C$ such that $\Re(i\lambda)>0$ 
we have 
\begin{equation*}
\lim_{t\to\infty}{\rm e}^{2(\rho-\Re(i\lambda))t} \|\Phi_{\lambda,\delta}(a_t)\|^2_{\rm HS}=d_{p,q}^2 |c_q(\lambda,p)|^2 d_{\sigma_q}.
\end{equation*}
\end{lemma}
 \begin{proof}  
Recall that $\Phi_{\lambda,\delta}(a_t)= \mathcal{P}_{q,\lambda}^p (P_\delta(\delta^{-1}(\cdot)))(a_t)$. Then
 \begin{eqnarray*}
 {e}^{2(\rho-\Re(i\lambda))t}\| \Phi_{\lambda,\delta}(a_t) \|^2_{\rm HS}
 &=&\sum_{j=1}^{d_\delta} \| {e}^{(\rho-\Re(i\lambda))t}  \Phi_{\lambda,\delta}(a_t)v_j \|^2_{V_{\tau_p} },\\
 &=&\sum_{j=1}^{d_\delta} 
 \| {e}^{(\rho-\Re(i\lambda))t} 
  \mathcal{P}_{q,\lambda}^p(P_\delta (\delta^{-1}(\cdot) v_j))(a_t)   \|^2_{V_{\tau_p}  }.
 \end{eqnarray*}
Using Proposition  \ref{Fatou} and \autoref{lemma-c-f-2}, we obtain
 \begin{equation*}
 \begin{split}
 \lim_{t\to\infty}{e}^{2(\rho-\Re(i\lambda))t}\| \Phi_{\lambda,\delta}(a_t) \|^2_{\rm HS}
 &= d_{p,q}^2 \sum_{j=1}^{d_\delta} \langle \mathbf c(\lambda,\tau_p) \boldsymbol \ib_q^p P_\delta v_j,  \mathbf c(\lambda,\tau_p) \boldsymbol \ib^p_q P_\delta v_j \rangle_{V_{\tau_p}}.\\
 &=d_{p,q}^2 \|   \mathbf c(\lambda,\tau_p) \boldsymbol \ib_q^p P_\delta\|_{\rm HS} \\
  &= d_{p,q}^2  |c_{q}(\lambda, p)|^2  d_{\sigma_q}.
  \end{split}
 \end{equation*}
 \end{proof}

\begin{theorem}[Inversion formula]\label{inversion} Let $\tau_p\in\{\tau_1,\cdots,\tau_{\frac{n-1}{2}},   \tau^\pm_{\frac{n}{2}}\}$ and $\sigma_q\in \widehat{M}(\tau_p)$.
Assume   $\lambda\in \mathbb C$ such that 
\begin{equation*} 
\begin{cases}
\Re(i\lambda)>0&\text{ if $q=p$},\\ 
 \Re(i\lambda)>0 \, \text{ and } \; i\lambda\not=\rho-p+1& \text{ if $q=p-1$}. 
 \end{cases}
 \end{equation*}%
Let $F\in \mathcal{E}^{2}_{\sigma_q,\lambda}(G/K ; \tau_p)$ and let $f\in L^2(K/M;\sigma_q)$ be its boundary value. Then the following inversion formula holds in $L^2(K/M;\sigma_q)$
\begin{equation*}
f(k) = {d_{p,q}^{-1}} |c_q(\lambda,p)|^{-2}   \, \lim_{t\to\infty} e^{2(\rho- \Re(i\lambda) )t} \boldsymbol\pi_p^q \left(\int_K P_{p,\lambda}(ha_t,k)^* F(ha_t)\, {\rm d}h\right),
\end{equation*} 
  where $P_{p,\lambda}$ is the Poisson kernel given by 
 $P_{p, \lambda} (g, k)=\mathrm{e}^{-(i \lambda+\rho) H\left(g^{-1} k\right)} \tau_p\left(\kappa\left(g^{-1} k\right)\right)$. 
  \end{theorem}

\begin{proof}
Let $F\in \mathcal{E}^{2}_{\sigma_q,\lambda}(G/K ; \tau_p)$. By Theorem 
\ref{th-cas-L2}, there exists a unique $f\in L^2(K/M;\sigma_q)$ such that $F=\mathcal{P}_{q,\lambda}^p f$. 
Write $$f(k)=\sum_{\delta\in\widehat K(\sigma_q)}  \sum_{j=1}^{d_\delta} a^\delta_{j}P_\delta (\delta(k^{-1}))v_j.$$ Then
\begin{eqnarray*}
F(ka_t)
&=& \sum_\delta \sum_{j} a^\delta_{j} \Phi_{\lambda,\delta} (a_t)\delta(k^{-1}) v_j,
\end{eqnarray*} 
and therefore
\begin{eqnarray*}
\int_K \| F(ka_t)\|_{V_{\tau_p}  }^2\, {\rm d}k
&=& 
\sum_\delta \sum_{j} \frac{|a^\delta_{j}|^2}{d_\delta} \| \Phi_{\lambda,\delta} (a_t) \|^2_{\rm HS}.
\end{eqnarray*}
From Lemma \ref{key-lemma} we deduce
 \begin{equation*}
\lim_{t\to\infty} {\rm e}^{2(\rho-\Re(i\lambda))t} \int_K
 \| \mathcal{P}_{q,\lambda}^p f(ka_t)\|_{V_{\tau_p}  }^2\,{\rm d}k=d_{p,q}^2|c_q(\lambda,p)|^2 \| f\|^2_{L^2(K/M;\sigma_q)},
 \end{equation*}
which implies
 $$\lim_{t\to\infty}(g_t, \varphi)_{L^2(K/M;\,\sigma_q)}= ( f,   \varphi)_{L^2(K/M;\,\sigma_q)},\quad \forall \varphi\in L^2(K/M;\sigma_q),$$
 where $g_t$ is the $V_{\sigma_q}$-valued function defined by 
 $$g_t(k)={d_{p,q}^{-1}}|c_q(\lambda,p)|^{-2} {\rm e}^{2(\rho-\Re(i\lambda))t}\boldsymbol \pi_p^q \int_K P_{p,\lambda} (ha_t,k)^* F(ha_t)\,{\rm d}h.$$
 To obtain the inversion formula, it is only required to show that 
 $$ \lim_{t\to \infty} \| g_t\|_{L^2(K/M;\,\sigma_q)}=\|f\|_{L^2(K/M;\,\sigma_q)}.$$
 To do so, let us first   compute the Fourier coefficients  $c_{j}^\delta(g_t)$ of $g_t$: 
 \begin{eqnarray*}
 c_{j}^\delta(g_t)
 &=&\frac{d_\delta}{d_{\sigma_q}}\int_K\langle g_t(k), P_\delta\delta(k^{-1})v_j\rangle_{V_{\sigma_q}  }\, {\rm d}k\\
 &=&d_{p,q}^{-1}|c_q(\lambda,p)|^{-2} {\rm e}^{2(\rho-\Re(i\lambda))t}\\
 &\times &\frac{d_\delta}{d_{\sigma_q}} \sum_{\delta',\ell} a^{\delta'}_{\ell}\int_K 
 \big\langle \boldsymbol\pi_q^p\int_K P_{p,\lambda} (ha_t,k)^*\Phi_{\lambda,\delta'}(a_t)\delta'(h^{-1})v_\ell {\rm d}h, P_\delta\delta(k^{-1})v_j
 \big\rangle_{V_{\sigma_q}  }\, {\rm d}k.
 \end{eqnarray*}
Since $(\boldsymbol\pi_p^q)^*=\boldsymbol \ib_q^p$, we get
  \begin{eqnarray*}
 c_{j}^\delta(g_t)&=&   d_{p,q}^{-1}|c_q(\lambda,p)|^{-2}   {\rm e}^{2(\rho-\Re(i\lambda))t}\\
 &\times &\frac{d_\delta}{d_{\sigma_q}} \sum_{\delta',\ell} a^{\delta'}_{\ell}
 \int_K  \int_K
  \big\langle   \Phi_{\lambda,\delta'}(a_t)\delta'(h^{-1})v_\ell  ,P_{p,\lambda}(ha_t,k) \boldsymbol \ib_q^p P_\delta\delta(k^{-1})v_j
  \big\rangle_{V_{\tau_p}  }\,{\rm d}h {\rm d}k,\\
  \end{eqnarray*}
As $\displaystyle \int_K P_{p,\lambda}(ha_t,k) \boldsymbol \ib_q^pP_\delta\delta(k^{-1}) {\rm d}k= d_{p,q}^{-1}\Phi_{\lambda,\delta}(ha_t)$, we obtain
  \begin{eqnarray*}
  c_{j}^\delta(g_t) &=& d_{p,q}^{-2}|c_q(\lambda,p)|^{-2}  {\rm e}^{2(\rho-\Re(i\lambda))t}
 \frac{d_\delta}{d_{\sigma_q}} \sum_{\delta',\ell} a^{\delta'}_{\ell}\int_K \langle  \Phi_{\lambda,\delta'}(a_t)\delta'(h^{-1})v_\ell,   \Phi_{\lambda,\delta}(a_t)\delta(h^{-1})v_j\rangle_{V_{\tau_p}  }\, {\rm d}h,\\
  &=& d_{p,q}^{-2}|c_q(\lambda,p)|^{-2}  {\rm e}^{2(\rho-\Re(i\lambda))t}
 \frac{d_\delta}{d_{\sigma_q}} \sum_{\delta',\ell} a^{\delta'}_{\ell} 
  \int_K \langle \delta(h) \Phi_{\lambda,\delta}(a_t)^* \Phi_{\lambda,\delta'}(a_t) \delta'(h^{-1}) v_\ell,v_j\rangle_{V_{\tau_p}  } \, {\rm d}h.
    \end{eqnarray*}
By the Schur lemma, we get
  \begin{eqnarray*}
   c_{j}^\delta(g_t)  &=&d_{p,q}^{-2}|c_q(\lambda,p)|^{-2}  {\rm e}^{2(\rho-\Re(i\lambda))t}
\frac{d_\delta}{d_{\sigma_q}} \sum_{\ell} a^{\delta}_{\ell} \int_K \frac{1}{d_\delta}
  \tr\left(\Phi_{\lambda,\delta}(a_t)^* \Phi_{\lambda,\delta}(a_t)\right) \langle v_\ell,v_j\rangle_{V_\delta}\, {\rm d}h, \\
  &=&d_{p,q}^{-2}|c_q(\lambda,p)|^{-2}  {\rm e}^{2(\rho-\Re(i\lambda))t} \frac{1}{d_{\sigma_q}}  a^{\delta}_{_j}
 \|  \Phi_{\lambda,\delta}(a_t) \|^2_{\rm HS}.
 \end{eqnarray*}
From all the above computations, we conclude  that,  
 \begin{eqnarray*}
  \|g_t\|^2_{L^2(K/M,\sigma_q)}
 &=& \left({\rm e}^{2(\rho-\Re(i\lambda))t}|d_{p,q}c_q(\lambda,p)|^{-2}\right)^2 
 \sum_\delta \frac{d_{\sigma_q}}{d_\delta} \sum_{j} \frac{1}{d_{\sigma_q}^2}       |a^\delta_{j}|^2 \| \Phi_{\lambda,\delta}(a_t) \|^4_{\rm HS},
  \end{eqnarray*}
  and by Lemma \ref{key-lemma} we get
   \begin{eqnarray*}
  \lim_{t\to\infty} \|g_t\|^2_{L^2(K/M;\sigma_q)}
=  \sum_\delta \frac{d_{\sigma_q}}{d_\delta} \sum_{j} |a^\delta_{j}|^2
=  \|f\|^2_{L^2(K/M;\sigma_q)}.
  \end{eqnarray*}
  To finish the proof, we have to justify that we can reverse $\lim_{t\to\infty}$ and $\sum_\delta$ by proving  that the series
  $$
  \sum_{\delta \in \widehat K(\sigma_q)}\frac{1}{d_\delta  } \sum_{j=1}^{d_\delta}        |a^\delta_{j}|^2 \| \Phi_{\lambda,\delta}(a_t) \|^4_{\rm HS},
  $$
  is uniformly convergent. This follows easily from Lemma \ref{lemma-lim-sum}.
     
\end{proof}

\section{The $L^r$-range of the Poisson transform}\label{sec6}
In this section, we will use  the $L^2$-characterization established in Theorem \ref{th-cas-L2} and the inversion formula proved in Theorem \ref{inversion} to prove  our main theorem :

\begin{theorem}\label{cas-Lr} Let  $\tau=\tau_p\in\{\tau_1,\cdots,\tau_{\frac{n-1}{2}},   \tau^\pm_{\frac{n}{2}}\}$ and $\sigma=\sigma_q\in \widehat{M}(\tau)$ with $q=p-1,p$.
 Assume $\lambda\in \mathbb C$ such that  
 \begin{equation*} 
\begin{cases}
\Re(i\lambda)>0&\text{ if $q=p$},\\ 
 \Re(i\lambda)>0 \, \text{ and } \; i\lambda\not=\rho-p+1& \text{ if $q=p-1$}. 
 \end{cases}
 \end{equation*}  
  For $1<r<\infty$, 
 the Poisson transform $\mathcal{P}_{\sigma,\lambda}^\tau$ is  a topological isomorphism
mapping  the space $L^r(K/M;\sigma)$ onto $\mathcal{E}^r_{\sigma_q,\lambda}(G/K ; \tau)$.
Furthermore,   there exists a positive constant $\gamma_\lambda$ such that   for every $f\in L^r(K/M;\sigma)$,
\begin{equation*}\label{esti-F}
d_{\tau,\sigma} | c_\sigma(\lambda,\tau)| \| f\|_{ L^r (K/M;\,\sigma)}\leqslant
 \|\mathcal{P}_{\sigma,\lambda}^\tau f\|_{r, \lambda } \leqslant  d_{\tau,\sigma}\,\gamma_\lambda \| f\|_{ L^r (K/M;\,\sigma)},
\end{equation*}
 where $ | c_\sigma(\lambda,\tau)|$ are  given explicitly in Proposition \ref{c-explicit} and Proposition \ref{c-explicit-special}.
\end{theorem}
 
\begin{proof} We will present the proof specifically for the generic case; similar arguments apply to the special cases, which we leave to the reader.

The necessary condition follows from  Theorem \ref{gaillard},  Corollary \ref{bijective1}, Corollary \ref{cor-olbrich}  and Proposition \ref{nec-cdt}. 
For the  sufficiency condition, let $F\in \mathcal{E}^{r}_{\sigma_q,\lambda}(G/K ; \tau_p)$ and write $F(g)=\sum_i F_i(g) u_i$ where $(u_i)_i$ is an orthonormal basis of $\bigwedge^p\mathbb C^n$. Fix  $(\chi_m)_m$ to be an approximation of the identity   in $C^\infty(K)$ and let 
$F_{i,m}(g)=\int_K\chi_m(k)F_i(k^{-1}g){\rm d}k$. Then $(F_{i,m})_m$ converges point-wise to $F_i$. Define $F_m : G\to \bigwedge^p\mathbb C^n$ by $F_m(g)=\sum_i F_{i,m}(g) u_i$. Then
\begin{eqnarray*}
F_m(g)&=&\sum_i \left(\int_K\chi_m(k)F_{i}(k^{-1}g){\rm d}k\right) u_i,\\
&=& \int_K \chi_m(k) \sum_i F_{i}(k^{-1}g) u_i {\rm d}k,\\
&=& \int_K \chi_m(k) F(k^{-1} g) {\rm d}k.\\
\end{eqnarray*}
We have $\|F_m(g) -F(g)\|^2_{\bigwedge^p\mathbb C^n}  {\to 0}$ as  ${m\to\infty}$,     
then $F_m\in 
\mathcal{E}_{\sigma_q,\lambda}(G/K ; \tau_p)$   for every $m$.
Further,
\begin{eqnarray*}
 F_m(ka_t)&=&\int_K \chi_m(h) F(h^{-1} ka_t) {\rm d}h,\\
&=&\left( \chi_m\ast F^t\right)(k),
\end{eqnarray*}
where $F^t : K\to \bigwedge^p\mathbb C^n$ is defined for any $t>0$ by $F^t(g) =F(ga_t)$. Since 
$$ \| (\chi_m\ast F^t)(k)\|_{\bigwedge^p\mathbb C^n}
\leqslant \int_K
 |\chi_m(h) |  \| F^t(h^{-1}k)\|_{\bigwedge^p\mathbb C^n}   {\rm d}h,$$
it follows 
$$  \| F_m^t(k)\|_{\bigwedge^p\mathbb C^n} \leqslant  \left(|\chi_m (\cdot)|  \ast \| F^t(\cdot)\|_{\bigwedge^p\mathbb C^n}\right)(k).$$
Therefore 
\begin{eqnarray*}
 \| F_m^t\|_{L^r(K;\bigwedge^p\mathbb C^n)}
 &\leqslant& \left\| |\chi_m (\cdot)|  \ast \| F^t(\cdot)\|_{\bigwedge^p\mathbb C^n}\right\|_{L^r(K)}.
 \end{eqnarray*}
Applying Young's involution inequality, we obtain  
 \begin{eqnarray}\label{holder1}
 \| F_m^t\|_{L^r(K; \bigwedge^p\mathbb C^n)} &\leqslant& 
 \|\chi_m\|_{L^1(K)} \; \left\| \| F^t(\cdot)\|_{\bigwedge^p\mathbb C^n}\right\|_{L^r(K)} \nonumber,\\
 &=& 
  \|F^t\|_{L^r(K; \bigwedge^p\mathbb C^n)},
 \end{eqnarray}
 and
 \begin{eqnarray}\label{holder2}
 \| F_m^t\|_{L^2(K; \bigwedge^p\mathbb C^n)} &\leqslant& \|\chi_m\|_{L^2(K)} \; \left\| \| F^t(\cdot)\|_{\bigwedge^p\mathbb C^n}\right\|_{L^1(K)},\nonumber\\
 &=& \|\chi_m\|_{L^2(K)} \;  \| F^t\|_{L^r(K;{\bigwedge^p\mathbb C^n})}.
 \end{eqnarray}
The inequality \eqref{holder2}  implies 
$$\sup_{t>0} {\rm e}^{(\rho-\Re(i\lambda))t} \left(\int_K \|F_m(ka_t)\|_{\bigwedge^p\mathbb C^n}^2 \right)^{1/2}\leqslant  \|\chi_m\|_{L^2(K)}   \| F\|_{ {r,\lambda}}<\infty.$$
Hence, for each $m$,  $F_m\in \mathcal{E}^{2}_{\sigma_q,\lambda}(G/K ; \tau_p)$ and from Theorem \ref{th-cas-L2} it follows that there exists $f_m\in L^2(K/M;\sigma_q)$ such that $F_m=\mathcal{P}_{q,\lambda}^p f_m$. 
Now, we need to prove that $f_m\in L^r(K/M;\sigma_q)$.  
According to Theorem \ref{inversion} we have, for any $\varphi\in C^\infty(K/M;\sigma_q)$,
$$\int_K \langle f_m(k),\varphi(k)\rangle_{\bigwedge^q\mathbb C^{n-1}}{\rm d}k=\lim_{t\to \infty} 
\int_K \langle g_m^t(k), \varphi(k)\rangle_{\bigwedge^q\mathbb C^{n-1}} {\rm d}k,$$
where
$$g_m^t(k):=d_{p,q}^{-2} |c_q(\lambda,p)|^{-2} e^{2(\rho-\Re(i\lambda))t} \boldsymbol\pi^q_p\int_K P_{p,\lambda}(ha_t,k)^* F_m(ha_t){\rm d}h.$$
Further,
 \begin{eqnarray*}
 &&\int_K \langle g_m^t(k),\varphi(k)\rangle_{\bigwedge^q\mathbb C^{n-1}} {\rm d} k\\
&&= d_{p,q}^{-2} |c_q(\lambda,p)|^{-2} e^{2(\rho-\Re(i\lambda))t} \int_K \langle \boldsymbol \pi^q_p\int_K P_{p,\lambda}(ha_t,k)^* F_m(ha_t){\rm d}h,\varphi(k)\rangle_{\bigwedge^q\mathbb C^{n-1}} {\rm d}k,\\
&&=d_{p,q}^{-2} |c_q(\lambda,p)|^{-2} e^{2(\rho-\Re(i\lambda))t}  \int_K\int_K \langle F_m(ha_t), P_{p,\lambda}(ha_t,k)\boldsymbol  \ib^p_q \varphi(k) {\rm d}k\rangle_{\bigwedge^p\mathbb C^n} {\rm d}h, \\
&&= d_{p,q}^{-3}  |c_q(\lambda,p)|^{-2} e^{2(\rho-\Re(i\lambda))t} \int_K \langle F_m(ha_t), (\mathcal{P}_{q,\lambda}^p \varphi)(ha_t) \rangle_{\bigwedge^p\mathbb C^n} {\rm d}h.
 \end{eqnarray*}
It follows  then that  
 \begin{eqnarray*}
 &&\left|\int_K \langle g_m^t(k),\varphi(k)\rangle_{\bigwedge^q\mathbb C^{n-1}}{\rm d}k\right| \\
 && \leqslant d_{p,q}^{-3}  |c_q(\lambda,p)|^{-2} e^{2(\rho-\Re(i\lambda))t}  \int_K \|F_m(ha_t)\|_{\bigwedge^p\mathbb C^n} \|\mathcal{P}_{q,\lambda}^p\varphi (ha_t)\|_{\bigwedge^p\mathbb C^n} {\rm d}h.
 \end{eqnarray*}
By H\"older's inequality (with $\frac{1}{r}+\frac{1}{s}=1$), we deduce 
 \begin{eqnarray*}
  \left|\int_K \langle g_m^t(k),\varphi(k)\rangle_{\bigwedge^q\mathbb C^{n-1}}{\rm d}k\right|
  \leqslant d_{p,q}^{-3}  |c_q(\lambda,p)|^{-2} e^{2(\rho-\Re(i\lambda))t} \|F_m^t\|_{L^r(K;\bigwedge^p\mathbb C^{n})} \| (\mathcal{P}_{q,\lambda}^p\varphi)^t\|_{L^s(K;\bigwedge^p\mathbb C^n)},
 \end{eqnarray*}
 where $(\mathcal{P}_{q,\lambda}^p\varphi)^t(k)=(\mathcal{P}_{q,\lambda}^p\varphi)(ka_t)$.
Using  \autoref{holder1}  and Proposition \ref{pro-gamma-lambda} we get
 \begin{eqnarray*}\label{eq-functional1}
 \left|\int_K \langle f_m(k),\varphi(k)\rangle_{\bigwedge^q\mathbb C^{n-1}}{\rm d}k\right| 
 &\leqslant& 
 \gamma_\lambda d_{p,q}^{-2}  |c_q(\lambda,p)|^{-2}   \|F_m\|_{r,\lambda} \|  \varphi\|_{L^s(K/M;\,\sigma_q)},\\
 &\leqslant& 
 \gamma_\lambda d_{p,q}^{-2} |c_q(\lambda,p)|^{-2}   \|F\|_{{r,\lambda}} \|  \varphi\|_{L^s(K/M; \,\sigma_q)}.
 \end{eqnarray*}
By taking the supremum over all $\varphi\in C^\infty(K/M;\sigma_q)$ with $\|  \varphi\|_{L^s(K/M;\sigma_q)}=1$ we obtain
$$\|f_m\|_{L^r(K/M;\, \sigma_q)} \leqslant \gamma_\lambda d_{p,q}^{-2}  |c_q(\lambda,p)|^{-2}  \|F\|_{{r,\lambda}} , $$
which implies $f_m$, initially belongs to  $L^2(K/M;\sigma_q)$, is in fact in   $L^r(K/M;\sigma_q)$.

For every $m$, define the linear form $T_m$ on $L^s(K/M;\sigma_q )$ by
$$T_m(\varphi)=\int_K\langle f_m(k),\varphi(k) \rangle_{\bigwedge^q\mathbb C^{n-1}} {\rm d}k.$$
Clearly, $T_m$ is continuous and 
 \begin{eqnarray*}
 |T_m(\varphi)| 
  &\leqslant& \gamma_\lambda d_{p,q}^{-2}  |c_q(\lambda,p)|^{-2}   \|F\|_{r,\lambda} \|  \varphi\|_{L^s(K/M;\, \sigma_q)}.
  \end{eqnarray*}
This shows that  $(T_m)_m$ is uniformly bounded in $L^s(K/M;\sigma_q)$, with
$$\sup_m \|T_m\|_{\rm{op}} \leqslant \gamma_\lambda d_{p,q}^{-2}  |c_q(\lambda,p)|^{-2}   \|F\|_{{r,\lambda}}\;. $$
The Banach-Alaouglu-Bourbaki theorem will then ensures the existence of a subsequence of bounded operators $(T_{m_j})$ which converges to a bounded operator $T$ under the weak-$\ast$ topology, with
$$  \|T\|_{\rm{op}} \leqslant \gamma_\lambda d_{p,q}^{-2}  |c_q(\lambda,p)|^{-2}   \|F\|_{{r,\lambda}}\;.$$
Thus,   Riesz's representation theorem guarantees the existence of  a unique $f\in L^r(K/M;\sigma_q)$ such that
$$T(\varphi) =\int_K \langle   \varphi(k), f(k) \rangle_{\bigwedge^q\mathbb C^{n-1}} {\rm d}k.$$
We consider  the test function $\varphi_g(k)= P_{p,\lambda}(g,k) v$ with $v\in\bigwedge^p\mathbb C^n$ and where $P_{p, \lambda}(g, k)=\mathrm{e}^{-(i \lambda+\rho) H\left(g^{-1} k\right)} \tau_p\left(\kappa\left(g^{-1} k\right)\right)$
 is the Poisson kernel.  Then 
$$T(\varphi_g) =\langle v, \mathcal{P}_{q,\lambda}^pf(g)\rangle_{\bigwedge^p\mathbb C^{n}}.$$ On the other hand 
$$T_{m_j}(\varphi_g) =\langle v, \mathcal{P}_{q,\lambda}^pf_{m_j}(g)\rangle_{\bigwedge^p\mathbb C^{n}} = \langle v, F_{m_j}(g)\rangle_{\bigwedge^p\mathbb C^{n}}.$$ 
Taking the limit of the above identity when $j\to\infty$ we conclude that   
$F(g)=\mathcal{P}_{q,\lambda}^p f (g)$ for every $g\in G$.
\end{proof}

As an immediate consequence of Theorem \ref{cas-Lr}  we obtain the following  characterization of co-closed    harmonic $p$-forms on $H^n(\mathbb R)$:

\begin{corollary}\label{corr-harm}  Let $p$ be an integer with $0\leqslant p<(n-1)/2$. 
   For $1<r<\infty$, the Poisson transform $\mathcal P_{p,i(p-\rho)}^p$ is a topological isomorphism   from the space $L^r(K/M; \sigma_p)$ onto the space $\mathcal{E}^r_{p,i(p-\rho)}(G/K ; \tau_p)$.  Moreover, for every $f\in L^r (K/M;\,\sigma_p)$ the following estimates hold,
\begin{align*}
\frac{2(\rho-p)}{2\rho-p} c_p(\rho) \| f\|_{ L^r (K/M;\,\sigma_p)}\leqslant
 \|\mathcal{P}_{p,i(p-\rho)}^p f\|_{{r,i(p-\rho)}} \leqslant  c_p(\rho) \| f\|_{ L^r (K/M;\,\sigma_p)},
\end{align*}
where 
$$
c_p(\rho)=d_{p,p}\frac{2^p\Gamma(\rho+\frac{1}{2})\Gamma(\rho-p)}{\Gamma(\rho-\frac{p}{2})\Gamma(\rho-\frac{p}{2}+\frac{1}{2})}.
$$   
\end{corollary}
In the case where $p=0$, we recover the classical fact that the Poisson transform is an isometric isomorphism from $L^r(\partial H^n(\mathbb R))$ onto the Hardy-harmonic space on $H^n(\mathbb R)$ (see \cite{stoll}).

 \pdfbookmark[1]{References}{ref}

\end{document}